\documentclass{amsart}
\usepackage[centertags]{amsmath}
\usepackage{amsfonts}
\usepackage{amssymb}
\usepackage{amsthm,mathtools}
\usepackage{amsrefs}
\usepackage{pifont}
\usepackage[dvipsnames]{xcolor}
\usepackage{graphicx}
\usepackage{comment}
\usepackage{float}

\usepackage[english]{babel}

\makeatletter
\g@addto@macro\bfseries{\boldmath} 
\makeatother

\usepackage[dvipsnames]{xcolor}

\usepackage[breaklinks=true,colorlinks=true,
linkcolor=MidnightBlue,citecolor=MidnightBlue,
urlcolor=MidnightBlue]{hyperref}

\usepackage{tikz}
\usepackage{tikz-cd}
\usepackage[normalem]{ulem}

\usepackage[shortlabels]{enumitem} 
\usepackage{bm}

\setlist[enumerate,1]{label=\textup{(\alph*)}}
\usepackage{pifont}

\newcommand{\str}{\operatorname{str-exp}}
\newcommand{\smo}{\operatorname{smo}}

\newcommand{\aconv}{\operatorname{aconv}}

\newcommand{\id}{\operatorname{Id}}

\newcommand{\K}{\mathbb{K}}
\newcommand{\N}{\mathbb{N}}
\newcommand{\T}{\mathbb{T}}
\newcommand{\R}{\mathbb{R}}

\newcommand{\C}{\mathbb{C}}
\newcommand{\Lin}{\mathcal{L}}
\newcommand{\ELL}{\mathcal{L}}
\newcommand{\DP}{\mathcal{DP}}
\newcommand{\KP}{\mathcal{K}}
\newcommand{\FR}{\mathcal{FR}}

\newcommand{\eps}{\varepsilon}

\DeclareMathOperator{\supp}{supp}
\DeclareMathOperator{\dist}{dist}

\DeclareMathOperator{\re}{Re}

\DeclareMathOperator{\NA}{NA}

\DeclareMathOperator{\RSE}{RSE}

\DeclareMathOperator{\ASE}{ASE}
\DeclareMathOperator{\SE}{SE}
\DeclareMathOperator{\Fl}{Fl}

\theoremstyle{plain}
\newtheorem{theorem}{Theorem}[section]
\newtheorem{cor}[theorem]{Corollary}
\newtheorem{prop}[theorem]{Proposition}
\newtheorem{lem}[theorem]{Lemma}

\theoremstyle{definition}

\newtheorem{example}[theorem]{Example}
\newtheorem{question}[theorem]{Question}
\newtheorem{rem}[theorem]{Remark}
\newtheorem{defi}[theorem]{Definition}

\newtheorem{examples}[theorem]{Examples}

 \makeatletter
\renewcommand{\tocsection}[3]{%
	\indentlabel{\@ifnotempty{#2}{\bfseries\ignorespaces#1 #2\quad}}\bfseries#3}
\renewcommand{\tocsubsection}[3]{%
	\indentlabel{\@ifnotempty{#2}{\ignorespaces#1 #2\quad}}#3}

\newcommand\@dotsep{4.5}
\def\@tocline#1#2#3#4#5#6#7{\relax
	\ifnum #1>\c@tocdepth 
	\else
	\par \addpenalty\@secpenalty\addvspace{#2}%
	\begingroup \hyphenpenalty\@M
	\@ifempty{#4}{%
		\@tempdima\csname r@tocindent\number#1\endcsname\relax
	}{%
		\@tempdima#4\relax
	}%
	\parindent\z@ \leftskip#3\relax \advance\leftskip\@tempdima\relax
	\rightskip\@pnumwidth plus1em \parfillskip-\@pnumwidth
	#5\leavevmode\hskip-\@tempdima{#6}\nobreak
	\leaders\hbox{$\m@th\mkern \@dotsep mu\hbox{.}\mkern \@dotsep mu$}\hfill
	\nobreak
	\hbox to\@pnumwidth{\@tocpagenum{\ifnum#1=1\bfseries\fi#7}}\par
	\nobreak
	\endgroup
	\fi}
\AtBeginDocument{%
	\expandafter\renewcommand\csname r@tocindent0\endcsname{0pt}
}
\def\l@subsection{\@tocline{2}{0pt}{2.5pc}{5pc}{}}
\makeatother

\begin{document}
\title[Bollobás-type theorems for RSE operators]{Bollobás-type theorems for range strongly exposing operators} 

\author[del R\'{\i}o]{Helena del R\'{\i}o} 
\address[]{Department of Mathematical Analysis and Institute of Mathematics (IMAG), University of Granada, E-18071 Granada, Spain \newline
	\href{https://orcid.org/0009-0004-5078-6993}{ORCID: \texttt{0009-0004-5078-6993} }}
\email[]{\texttt{helenadelrio@ugr.es}}

\date{\today}



\begin{abstract} We study Bollobás-type theorems for range strongly exposing operators. When such a theorem holds for operators from a Banach space $X$ into another Banach space $Y$, we say that the pair $(X,Y)$ satisfies the Bishop-Phelps-Bollobás property for range strongly exposing operators (BPBp-RSE, for short). We provide new characterisations of uniform convexity and complex uniform convexity via the BPBp-RSE, including for pairs involving spaces such as $L_1(\mu), L_\infty(\mu)$ and $c_0$. In particular, we show that $(L_1(\mu), Y)$ satisfies the BPBp-RSE if and only if $Y$ is uniformly convex, and that $(L_\infty(\mu), Y)$ or $(c_0, Y)$ satisfy the BPBp-RSE if and only if $Y$ is $\mathbb{C}$-uniformly convex. We also highlight differences between the real and complex cases, showing that there exist pairs $(X, Y)$ for which the BPBp-RSE holds in the complex setting but fails for their respective underlying real spaces. Additionally, we consider various subspaces of operators, such as compact and finite-rank, and extend several results from the literature to this new setting. The paper concludes with a collection of open problems.

\end{abstract}

\maketitle
 {\parskip=0ex 	\tableofcontents }

\section{Introduction}

Let $X$ be a Banach space over the field of scalars $\K$, which can be either $\R$ or $\C$. We will denote by $B_X$ and $S_X$ the unit ball and the unit sphere of $X$, respectively, and we will write $\T:= S_{\K}$. $X^*$ will stand for the topological dual of $X$ and $\ELL(X, Y)$ will denote the Banach space of all bounded linear operators from $X$ to another Banach space $Y.$ An operator $T\in\ELL(X, Y)$ is said to \emph{attain its norm} if there is $x\in S_X$ such that $$\|Tx\|=\|T\|=\sup_{z\in B_X} \|Tz\|.$$
In this case, we say that $Y$ is \emph{norm-attaining}. The set of all norm-attaining operators between $X$ and $Y$ will be denoted by $\NA(X, Y).$ When $Y=\K,$ we will simply write $\NA(X).$

The classical Bishop-Phelps theorem states that the set of norm-attaining functionals is dense in the dual of every Banach space \cite{BP1961}. They asked whether this result could be extended to the vector valued case, which was answered negatively by Lindenstrauss in \cite{L1963}. He also provided some conditions assuring the denseness of norm-attaining operators between two Banach spaces; for instance, when the domain is reflexive or when the range satisfies the so-called property $\beta.$ Since then, an extensive research has been carried out with both positive and negative results. Some classical and recent results can be found in \cite{ACKP1998, Bourgain1970, JW1979, JMR2023, S1983, Uhl}. We also refer the interested reader to the expository papers \cite{A2006, M2016} and the references therein.

The discoveries of Bourgain were specially relevant to the theory as they relate the denseness of norm-attaining operators to the Radon-Nikodým property (RNP, for short). For instance, he proved in \cite{Bourgain1970} that if a Banach space $X$ has the RNP, then the absolutely strongly exposing operators (which are norm-attaining) are dense in $\ELL(X, Y)$ for every Banach space $Y$. 
\begin{defi}
    We say that $T\in\ELL(X, Y)$ is \emph{absolutely strongly exposing} and write $T\in \ASE(X, Y)$ if there exists $x_0\in S_X$ such that whenever a sequence $(x_n) \subseteq B_X$ satisfies that $\|Tx_n\| \to \|T\|$, there exists a subsequence $(x_{\sigma(n)}) \subseteq (x_n)$ and some $\lambda \in \mathbb{T}$ such that $x_{\sigma(n)} \to \lambda x_0$. 
\end{defi}

This means that $T$ attains its norm at an element which is unique up to rotation, in the strong way that whenever the norm of the image of a point in $x\in B_X$ approaches the maximum then $x$ must be close to this element. We refer the reader to \cite{JMR2023} for some of the properties of $\ASE$ operators and extensions of Bourgain's work. Although useful, the concept of absolutely strongly exposing operators seems to be very restrictive. Indeed, the existence of an operator $T\in\ASE(X, Y)$ implies that the unit ball of $X$ has strongly exposed points, and the denseness of $\ASE(X, Y)$ in $\ELL(X, Y)$ implies such of $\SE(B_X),$ the set of strongly exposing functionals, in $X^*$ (for the definition of this and other standard concepts used in this paper, see Subsection \ref{sec:definitions}). In particular, this concept turns out to be useless when the domain does not have any strongly exposed points, which is the case of some classical Banach spaces including $L_1(\mu)$ when $\mu$ is atomless and $C(K)$ when $K$ is infinite. This is why it was recently introduced in \cite{CDFJM2025} the following class of operators, which is bigger than absolutely strongly exposing operators although smaller than classical norm-attaining operators. 

\begin{defi}
    We say that $T\in\ELL(X, Y)$ is \textit{range strongly exposing} and write $T\in \RSE(X,Y)$ if there exists $x_0 \in S_X$ such that whenever a sequence $(x_n) \subseteq B_X$ satisfies  $\|Tx_n\| \to \|T\|$, there exists a subsequence $(x_{\sigma(n)}) \subseteq (x_n)$ such that $Tx_{\sigma(n)} \to \lambda Tx_0$ for some $\lambda \in \mathbb{T}$.
\end{defi}

The main difference with ASE operators is that for $T\in\RSE(X, Y)$ it is only required that the maximum value is unique (up to rotation), but the maximum norm may be attained at many different points in the domain. Among the known results on RSE operators, we outline the following ones. $\RSE(L_1(\mu), Y)$ is dense in $\ELL(L_1(\mu), Y)$ whenever $Y$ has the RNP, and the converse is also true when $\mu$ is not purely atomic \cite[Corollary 3.8, Theorem 3.10]{CDFJM2025}; weakly compact operators from $C(K)$ can always be approximated by weakly compact RSE operators for any compact Hausdorff space $K$ \cite[Theorem 3.15]{CDFJM2025}, and for every infinite dimensional Banach space $Y$, there is $X$ such that $\RSE(X, Y)$ is not dense in $\ELL(X, Y)$ \cite[Theorem 3.1]{CDFJM2025}.

Coming back to the Bishop-Phelps theorem, in \cite[Theorem~1]{B1970}, Bollobás stated a quantitative version of this result which allows us approximating simultaneously the functional and the point where the norm is almost attained.
This result, known nowadays as the Bishop-Phelps-Bollobás theorem, can be stated (with a refinement from \cite[Corollary~2.4]{CKMMR2014}) as follows.
\begin{theorem}\label{teo:BPB}
    For every $0<\eps<2$, if $x_0\in B_X$ and $x_0^*\in B_{X^*}$ are such that $\re x_0^*(x_0)>1-\frac{\eps^2}{2}$, then there exist $x\in S_X$ and $x^*\in S_{X^*}$ such that
    \begin{equation*}
        x^*(x)=1, \quad \|x^*-x_0^*\|<\eps \quad \text{and} \quad \|x-x_0\|<\eps.
    \end{equation*}
\end{theorem}

The version for operators of this theorem was introduced in the 2008 paper \cite{AAGM2008} as a property of pairs. We say that a pair of Banach spaces $(X, Y)$ satisfies the \textit{Bishop-Phelps-Bollobás property} (BPBp, for short) if for every $\eps >0$, there exists $0<\eta (\eps)<\eps$ such that when $T\in\ELL(X, Y)$ with $\|T\|=1$ and $x\in S_X$ satisfy
    \begin{equation*}
        \|Tx\|>1-\eta(\eps),
    \end{equation*}
    there exist $S\in \ELL(X, Y)$ and $x_0\in S_X$ satisfying
    \begin{equation*}
        \|S\|=\|Sx_0\|=1, \quad \|S-T\|<\eps \quad\text{and} \quad \|x-x_0\|<\eps.
    \end{equation*}

The expository papers \cite{A2019, DGMR2022} account for the extensive research that has been made on this topic. 
\begin{examples}
    Among all the results in the literature, it is known that the pair $(X, Y)$ has the BPBp when
\begin{enumerate}\label{ex:BPBp}
    \item $X$ is uniformly convex and $Y$ is any Banach space \cite[Theorem 3.1]{KL2014} and \cite[Corollary 2.3]{ABGM2013},
    \item\label{item:ell_1} $X=\ell_1$ and $Y$ satisfies the so called AHSP \cite[Theorem 4.1]{AAGM2008} (indeed, this result is a characterisation) 
    \item $X=\ell_1^m$ for some $m\in \N$ and $Y$ has the AHSP (by item \ref{item:ell_1} and \cite[Theorem~2.1]{ACKLM2015}),
    \item $X=L_1(\mu)$ where $\mu$ is a $\sigma$-finite measure on an infinite $\sigma$-algebra and $Y$ has the RNP and the AHSP (in particular, when $Y$ is uniformly convex) \cite[Theorem~2.2]{CK2011},
    \item $X=C(K)$ and $Y$ is uniformly convex \cite[Theorem 2.2]{KL2015},
    \item $X=C_0(L)$ and $Y$ is $\C$-uniformly convex in the complex case \cite[Theorem 2.4]{A2016},
    \item $Y$ has Lindenstrauss' property $\beta$ and $X$ is any Banach space \cite[Theorem 2.2]{AAGM2008} and
    \item $X$ is Asplund and $Y\subseteq C(K)$ is a uniform algebra \cite[Theorem 3.6]{CGK2013}.
\end{enumerate}
\end{examples}

Moreover, if $\mathcal{M}$ is a subspace of $\ELL(X, Y),$ we will say that $\mathcal{M}$ satisfies the BPBp if the conditions in the definition of the BPBp are satisfied for every $T\in S_{\mathcal M}$ and the operator $S\in \NA(X, Y)$ which approximates $T$ also belongs to $\mathcal{M}.$ Among others, the papers \cite{ABGKM2014, ABCCKLLM014, ACK2011,  CGK2013, CGKS2018, DGMM2018} address the BPBp of some subspaces of operators including compact, weakly compact, and finite rank operators. In the same direction of research, we send the interested reader to the recent papers \cite{KLMW2020F, KLMW2020, M2016, R2017}.

It may be of interest to consider versions of the Bishop-Phelps-Bollobás property where the initial operator is approximated not only by a norm-attaining operator but by an operator with some special behaviour such as absolutely strongly exposing or range strongly exposing operators. It is worth mentioning that the study of a Bollobás type of approximation by ASE operators has already been addressed in the recent paper \cite{JMR2023}. For RSE operators, it is natural to consider the following property.

\begin{defi}\label{defi:BPBp-RSE}
    Let $X, Y$ be Banach spaces. The pair $(X, Y)$  is said to satisfy the \textit{Bishop-Phelps-Bollobás property for} $\RSE$ (\emph{BPBp-RSE}, for short) if for every $\eps >0$, there exists $0<\eta (\eps)<\eps$ such that for every $T\in\ELL(X, Y)$ with $\|T\|=1$ and every $x\in S_X$ such that 
    \begin{equation*}
        \|Tx\|>1-\eta(\eps)
    \end{equation*}
    there exist $S\in \RSE(X, Y)$ and $x_0\in S_X$ that satisfy
    \begin{equation*}
        \|S\|=\|Sx_0\|=1, \quad\|S-T\|<\eps \quad\text{and}\quad \|x-x_0\|<\eps.
    \end{equation*}
\end{defi}
In the cases where the operator $S$ can be chosen to belong to $\ASE(X, Y),$ we will say that $(X, Y)$ has the \textit{BPBp-ASE}. As it happens with the BPBp, these two properties depend strongly on the geometry of the unit balls of the two involved Banach spaces. 

If $\mathcal{M}$ is a subspace of $\ELL(X, Y),$ we will say that $\mathcal{M}$ satisfies the BPBp-RSE if the conditions in Definition \ref{defi:BPBp-RSE} are satisfied for every $T\in S_{\mathcal M}$ and the operator $S\in \RSE(X, Y)$ also belongs to $\mathcal{M}.$  

\begin{rem}
    The aim of this paper is to study which pairs of Banach spaces or subspaces of operators satisfy the BPBp-RSE and compare it to the Bishop-Phelps-Bollobás property in the classical sense. It is immediate that the BPBp-RSE implies both the denseness of $\RSE(X, Y)$ and the BPBp. However, the converse does not hold in general, as it will be shown in Example \ref{ex: ell_1^2}. Thus, the study of the BPBp for RSE operators is not a mere extension of the BPBp among the pairs of Banach spaces for which the $\RSE$ operators are dense.
\end{rem}

Let us present the outline of the paper. In Section \ref{sec:preliminaries} we recall some basic definitions and preliminary results that will be used throughout the paper, and we derive the first immediate consequence: if $X$ is uniformly convex, then the pair $(X, Y)$ has the BPBp-ASE for every Banach space $Y$, refining the results in \cite[Theorem 3.1]{KL2014} and \cite[Corollary 2.3]{ABGM2013}.

In Section \ref{sec: characterisations}, we show that the BPBp-RSE of a pair $(X, Y)$ provides several characterisations of uniform convexity and $\C$-uniform convexity. It should be pointed out that uniform convexity has previously been characterised by Bollobás type properties in \cite[Theorem 2.1]{KL2014}. More precisely, in Subsection \ref{sec:char L1}, we show that the pair $(L_1(\mu), Y)$ with $\mu$ being any $\sigma$-finite measure with $\dim (L_1(\mu))\geq 2$ has the BPBp-RSE if and only if $Y$ is uniformly convex. Let us mention that the BPBp of the pair $(\ell_1, Y)$ is characterised by the AHSP in \cite[Theorem 4.1]{AAGM2008}, and if $Y$ has the RNP, then the BPBp of the pair $(L_1(\mu), Y)$ is also equivalent to the AHSP of $Y$ \cite[Theorem 2.2]{CK2011} when $\mu$ is a $\sigma$ finite measure on an infinite $\Sigma$-algebra.
It is also known that $\RSE(L_1(\mu), Y)$ is dense in $\ELL(L_1(\mu), Y)$ if and only if $Y$ has the RNP when $\mu$ is not purely atomic \cite[Corollary 3.8, Theorem 3.10]{CDFJM2025}. 

Next, in Subsection \ref{sec: M summand complex} we prove that, in the complex case, $\C$-uniform convexity of $Y$ is a sufficient condition for the pair $(C_0(L), Y)$ to have the BPBp-RSE, where $L$ is any locally compact Hausdorff space (improving \cite[Theorem 2.4]{A2016}). In particular, we derive that the pair $(L_\infty(\mu), Y)$ has the BPBp in the complex case if and only if $Y$ is $\C$-uniformly convex. It turns out that the same characterisation also holds in the complex case when the domain is a $c_0$-sum of copies of a uniformly convex Banach space (improving \cite[Theorem 2.4]{CK2022}). Moreover, in Subsection \ref{sec: M summand real}, we show that, in the real case, uniform convexity of $Y$ characterises the BPBp-RSE of the pairs $(L_\infty(\mu), Y)$ (improving \cite[Theorem 5]{KLL2016}) and $(c_0(X), Y)$ for any uniformly convex Banach space $X.$ 

These results are surprising as the notions of uniform convexity and $\C$-uniform convexity are completely different. Observe that, for instance, that the pair $(c_0, \ell_1)$ satisfies the BPBp-RSE in the complex case, although the pair formed by their underlying real spaces, $(c_0(\ell_2^2),\ell_1(\ell_2^2))$, does not have this property (see Example \ref{ex: c0 ell_1}). 

In Section \ref{sec:subspaces of operators} we address the BPBp-RSE for some subspaces of operators, including finite rank and compact operators. First, in Subsection \ref{sec: char for compact} we show that the characterisations of uniform convexity and $\C$-uniform convexity also hold if we restrict ourselves to compact operators defined on $L_1(\mu)$ and $L_\infty(\mu)$, respectively. In order to do so, we extend some of the techniques in \cite{DGMM2018} for passing from sequence spaces to function spaces to the RSE setting, which will allow us to obtain some consequences for compact operators defined on $C_0(L)$ and isometric $\ell_1$-preduals. Secondly, in Subsection \ref{sec:ACK} we aim to extend the results in \cite{CGKS2018} about range spaces with ACK$\rho$ structure to the RSE setting. In this case, we need to have conditions on both the domain and the range spaces as shown by Example \ref{ex: ell_1^2}. Using these results, we are able to derive some consequences for compact operators whose range is an isometric $L_1$-predual, extending \cite[Theorem~4.2]{ABCCKLLM014} to the RSE setting. 

Finally, in Section \ref{sec:discussion} we define universality notions of the domain and the range for the BPBp-ASE and the BPBp-RSE as it was done for the classical BPBp in \cite{ACKLM2015} and derive some necessary conditions for being a universal BPBp-RSE domain or range. Namely, we show that if $X$ is a Banach space such that the pair $(X, Z)$ has the BPBp-RSE for every $Z$, then the strongly exposed points of $B_X$ are dense in $B_X$; and if $Y$ is such that the pair $(Z, Y)$ has the BPBp-RSE for every Banach space $Z$ then $Y$ must be uniformly convex and finite dimensional. At the end of this section we discuss some open questions that we find of interest for future research.

\section{Preliminaries}\label{sec:preliminaries}

\subsection{Notation and first definitions}\label{sec:definitions} 
Let us begin recalling some basic definitions. Given a non-empty bounded subset $C$ of a Banach space $X$, a point $x_0\in C$ is called a \emph{strongly exposed point of} $C$ if there is $x^*\in X^*\backslash \{0\}$ such that
\begin{equation*}
    \re x^*(x_0) = \sup_{x\in C} \re x^*(x)
\end{equation*} 
and for every sequence $(x_n)\subseteq C$ such that $\lim_n \re x^*(x_n)=\re x^*(x_0)$ we have that $(x_n)_{n\in \N}$ converges in norm to $x_0.$ In this case, we say that $x^*$ \emph{strongly exposes} $x_0$ in $C$ and that $x^*$ is a \emph{strongly exposing functional} of $C$. We will denote by $\str(C)$ and $\SE(C)$ the strongly exposed points of $C$ and its strongly exposing functionals, respectively. When the set is the unit ball of $X$ we clearly have that $\ASE(X, \K)=\SE(B_X).$

There is a duality between strong exposition of the unit ball of a Banach space and smoothness in its dual. We say that a point $x$ is a \emph{smooth point} (equivalently, $x$ is a G\^{a}teaux differentiable point) of $X$ if there exists a unique $x^*\in S_{X^*}$ such that $x^*(x)=\|x\|$, and write $\smo(X)$ for the set of all smooth points of $X$. If $\smo(X)=X$ we will simply say that $X$ is \emph{smooth}. By the \v{S}mulyan's lemma (see \cite[Corollary 7.22]{FHHMZ}), a point $x \in X$ is a G\^{a}teaux (resp.,\ Fr\'echet) differentiable point if and only if for every $(f_n)$ and $(g_n)$ in $B_{X^*}$ such that $f_n(x)\longrightarrow \|x\|$ and $g_n(x)\longrightarrow \|x\|$ one has that $(f_n-g_n)$ converges to $0$ in the weak-star topology of $X^*$ (resp.,\ in the norm topology). In particular, the set of Fr\'echet differentiable points of $X^*$ coincides with $\SE(B_X),$ the set of strongly exposing functionals of $B_X.$

A Banach space $X$ is \emph{strictly convex} if there is no nontrivial line segment contained in $S_X.$ As a consequence of \v{S}mulyan's lemma and Hahn-Banach theorem, this is equivalent to that $\NA(X, \K)\subseteq \smo(X^*).$ In particular, this happens when $X^*$ is smooth.

We say that $X$ is \textit{uniformly convex} if for every $\eps>0$ there exists $0<\delta <1$ such that if $x_1, x_2\in S_X$ satisfy $$\frac{\|x_1 +x_2\|}{2}>1-\delta$$ then $\|x_1-x_2\|<\eps.$ In such a case, the \emph{modulus of uniform convexity of} $X$ is defined by
\begin{equation*}
    \delta_X(\eps):=\inf\left\{1-\frac{\|x_1+x_2\|}{2}: x_1, x_2\in B_X, \|x_1-x_2\|\geq \eps\right\} \quad (\eps\in (0,2)).
\end{equation*}
With this notation, $X$ is uniformly convex if and only if $\delta_x(\eps)>0$ for every $\eps\in(0,2).$

A Banach space $Y$ is called \emph{uniformly smooth} if its norm is uniformly Fr\'echet differentiable on every $y\in S_Y$. We have that $Y$ is smooth if and only if $Y^*$ is uniformly convex, and uniformly convex spaces are reflexive, so this two notions are completely symmetric. Examples of Banach spaces which are both uniformly convex and uniformly smooth are $\ell_p$ and $L_p$ for $p\in (1, \infty).$

\begin{rem}\label{rem: unif convex implies strongly exposed}
    It is not difficult to show that if $X$ is uniformly convex, then every point $x\in S_X$ is strongly exposed by any $f\in S_{X^*}$ such that $f(x)=1.$ Moreover, this shows that $S_X$ is a set of uniformly strongly exposed points of $B_X$ and, as $X$ is reflexive, $\SE(B_X)=X^*$.
\end{rem}

There is a weaker notion than uniform convexity that will be useful in the complex case. A complex Banach space $Y$ is said to be $\C$-\textit{uniformly convex} if for every $\eps>0$, there exists $\delta>0$ such that whenever
$y_0, y_1\in Y$ are such that $\|y_0+wy_1\|\leq 1$ for every $w\in \T$ and $\|y_1\|>\eps$ we have $\|y_0\|<1-\delta.$ 
We will use the following characterisation from \cite{G1975}. For $\delta\geq 0$, let
\begin{equation*}
\begin{split}
    w_c(\delta)&:=\sup \{\|y\| : x\in S_Y, y\in Y \ \text{ with } \|x+wy\|\leq 1+\delta \quad \forall w\in \C \text{ with } |w|<1\}\\
    &=\sup \{\|y\| : x\in S_Y, y\in Y \ \text{ with } \|x+wy\|\leq 1+\delta \quad \forall w\in \T\}.
\end{split}
\end{equation*}
Then, $Y$ is $\C$- uniformly convex if and only if 
\begin{equation*}
    \lim_{\delta\rightarrow 0} w_c(\delta)=0.
\end{equation*}
The \emph{modulus of $\C$-convexity of} $Y$ is defined by
\begin{equation*}
    \delta_\C(\eps)=\inf_{x,y\in S_Y} \left\{\sup_{\lambda\in \T} \|x+\lambda\eps y\|-1 \right\} \quad (\eps>0).
\end{equation*}
From the definitions above, it is clearly satisfied that $Y$ is $\C$-uniformly convex if and only if $\delta_\C(\eps)>0$ for every $\eps>0.$ Uniform convexity implies $\C$-uniform convexity, but the converse is not true. An example of a $\C$-uniformly convex Banach space which is not uniformly convex is $\ell_1$ (see \cite{G1975}).

Finally, the following property was introduced in \cite[Definition~3.9]{JMR2023} and can be viewed as a version of the Bishop-Phelps-Bollobás theorem for strongly exposing functionals.

\begin{defi}
    We say that a Banach space $X$ satisfies \textit{property [P]} if for every $\eps >0$, there exists $\eta (\eps)>0$ such that for every $x^*\in S_{X^*}$ and every $x\in S_X$ satisfying 
    \begin{equation*}
        \re x^*(x)>1-\eta(\eps)
    \end{equation*}
    there exist $y^*\in\SE(B_X)$ and $y\in S_X$ such that
    \begin{equation*}
        y^*(y)=1, \quad \|x^*-y^*\|<\eps, \quad \|x-y\|<\eps.
    \end{equation*}
\end{defi}

This is clearly equivalent to the BPBp-ASE of the pair $(X, \K)$ (see Definition \ref{defi:BPBp-RSE}). In \cite[Proposition~3.11]{JMR2023} it is proved that $\str(B_X)=S_X$ is a sufficient condition to get this property so, in particular, uniformly convex spaces have it. For uniformly smooth spaces, a weaker sufficient condition still works. Indeed, we have the following result.
\begin{prop}\label{prop:unif smo}
    Let $X$ be a uniformly smooth Banach space. Then, $(X, \K)$ has property [P] if and only if $\overline{\str(B_X)}=S_X.$
\end{prop}
\begin{proof}
    One implication is clear from the definition of property [P]. To see the other one, suppose that $\overline{\str(B_X)}=S_X$ and let $\eps>0$. By \cite[Corollary~2.2]{KL2014}, there is $0<\eta(\eps)<1$ such that, for all $f\in B_{X^*}$ and all $x\in S_X$ satisfying $|f(x)|>1-\eta(\eps)$, there exists $f_0\in S_{X^*}$ with $|f_0(x)|=1$ and $\|f-f_0\|<\eps.$ Let $x\in S_X$ and  $x^*\in S_{X^*}$ be such that $$|x^*(x)|>1-\frac{\eta(\eps)}{2}.$$ There exists $y\in\str(B_X)$ such that $\|x-y\|<\min \{\frac{\eta(\eps)}{2}, \eps\}.$ Hence, $|x^*(y-x)|<\frac{\eta(\eps)}{2},$ so $$|x^*(y)|>|x^*(x)|-\frac{\eta(\eps)}{2}>1-\frac{\eta(\eps)}{2}-\frac{\eta(\eps)}{2}=1-\eta\left(\eps\right).$$
    Now, we can apply \cite[Corollary~2.2]{KL2014} to obtain $y_0^*\in S_{Y^*}$ such that $|y_0^*(y)|=1$ and $\|y_0^*-x^*\|<\eps$. By \cite[Lemma 1.4]{JMR2023}, we can find $y^*\in \SE(B_X)$ with $|y^*(y)|=1=\|y^*\|$ and $\|y^*-y_0^*\|<\eps,$ and therefore $\|x^*-y^*\|<2\eps.$
\end{proof}

\subsection{First results}
Now we will present some background on ASE and RSE operators and relate them to strongly exposed points and functionals. The following lemma is borrowed from 
\cite[Proposition~3.14]{CGMR2021}, \cite[Lemma~1.4]{JMR2023} and \cite[Lemma~2.6]{CDFJM2025}.

\begin{lem}[\mbox{\cite{CGMR2021,JMR2023, CDFJM2025}}]\label{lem:str-exp}
    Let $X, Y$ be Banach spaces and $T\in \ELL(X, Y).$ 
\begin{itemize}
    \item[(a)] If $T\in \ASE(X, Y)$ and $x_0\in S_X$ satisfies $\|Tx_0\|=\|T\|$, then $Tx_0\in \str((B_X))$ and it is strongly exposed by any $T^*y_0^*\in T^*(S_{Y^*})$ such that $\re y_0^*(Tx_0)=\|T\|.$
    \item[(b)] If $T\in \RSE(X, Y)$ and $x_0\in S_X$ satisfies $\|Tx_0\|=\|T\|$, then $Tx_0\in \str(\overline{T(B_X)})$ and it is strongly exposed by any $y_0^*\in S_{Y^*}$ such that $\re y_0^*(Tx_0)=\|T\|.$
\end{itemize}
\end{lem}

This result has some consequences for both the BPBp-RSE and the BPBp-ASE. In particular, item (a) implies that $B_X$ must have strongly exposed points when there exists some $T\in\ASE(X, Y)$. In \cite[Proposition~1.5]{JMR2023} it is proved that if $\ASE(X, Y)$ is dense in $\ELL(X, Y)$ for some $Y$, then $\SE(B_X)$ is dense in $X^*.$ An analogous argument shows that if the pair $(X, Y)$ has the BPBp-ASE, then $X$ must satisfy property [P].

\begin{prop}\label{prop:property P is necessary}
    Let $X$, $Y$ be two Banach spaces and let $\mathcal M\subseteq \ELL(X, Y)$ be any subspace containing all rank one operators. If $\mathcal{M}$ has the BPBp-ASE, then $X$ has property [P]. 
\end{prop}

\begin{proof}
    Let $\eps>0$ be given. Suppose that $\mathcal{M}$ satisfies the BPBp-ASE with some function $\eta(\eps)>0$ and let $x_0^*\in S_{X^*}$ and $x_0\in S_X$ be such that $\re x_0^*(x_0)>1-\eta(\eps).$
 Fix $y_0\in S_Y$ and define $T: X\longrightarrow Y$ by $Tx=x_0^*(x)y_0.$ Then, $T\in M, \ \|T\|=\|x_0^*\|=1$ and
    \begin{equation*}
        \|Tx_0\|=\|x_0^*(x_0)y_0\|=|x_0^*(x_0)|>1-\eta\left(\eps\right).
    \end{equation*}
    Therefore, there exist $S\in \ASE(X, Y)\cap M$ and $x_1\in S_X$ such that 
    \begin{equation*}
        \|S\|=\|Sx_1\|=1, \quad\|S-T\|<\eps  \ \text{ and } \ \|x_0-x_1\|<\eps.
    \end{equation*}
    Let $y_0^*\in S_{Y^*}$ be such that $y_0^*(Sx_1)=S^*y_0^*(x_1)=1.$ Then,  $\|S^*y_0^*\|=1$ and $S^*y_0^*\in \SE(B_X)$ by Lemma \ref{lem:str-exp}. As 
    $$\left|1-y_0^*(y_0)x_0^*(x_1)\right|=|y_0^*(Sx_1-Tx_1)|\leq \|y_0^*\|\|S-T\|\|x_1\|< \eps$$
    then $|y_0^*(y_0)|\geq|y_0^*(y_0)||x_0^*(x_1)|>1-\eps,$ so $|1-|y_0^*(y_0)||\leq \eps.$ Now, take $\theta\in \T$ so that $\theta y_0^*(y_0)=|y_0^*(y_0)|.$ Then, $\theta S^*y_0^*\in \SE(B_X), \|\theta S^*y_0^*\|=1=|\theta S^*y_0^*(x_1)|, \|x_1-x_0\|< \eps$ and 
    \begin{equation*}
        \|\theta S^*y_0^*-x_0^*\|\leq \|\theta S^*y_0^*-\theta y_0^*(y_0)x_0^*\|+\|\theta y_0^*(y_0)x_0^*-x_0^*\|<2\eps. \qedhere
    \end{equation*}
\end{proof}  

This means that, when looking for pairs $(X, Y)$ that have the BPBp-ASE, we must search among those in which $X$ satisfies property [P], which is very restrictive as it implies  both that $\str(B_X)$ and $\SE(B_X)$ are dense in $S_X$ and $X^*$, respectively.

As for RSE operators, we cannot expect an analogous to property [P] as in the case of functionals we have $\RSE(X, \K)=\NA(X, \K)$; thus, we can expect the BPBp-RSE to be less restrictive. However, item (b) in Lemma \ref{lem:str-exp} has the following interesting consequence which will be useful to obtain some characterisations. 

\begin{lem}[\mbox{\cite[Lemma~3.3]{CDFJM2025}}]\label{lemma:easy-lemma-to-kill-propertyRSE-B}
Let $X$ and $Y$ be Banach spaces. If $S\in \RSE(X,Y)$ attains its norm at $x_0\in S_X$ and there is $z\in X$ such that $\|x_0+z\|,\|x_0-z\|\leq 1$, then $S(z)=0$.
\end{lem}

The third ingredient will allow us to approximate norm-attaining operators by operators in RSE or ASE which are rank one perturbations of the original one and attain their norm at the same point. A proof of this proposition can be found in \cite[Proposition~3.14]{CGMR2021} and \cite[Lemma~2.6]{CDFJM2025}. 

\begin{prop}[\mbox{\cite{CGMR2021, CDFJM2025}}]\label{prop: aproximate NA by RSE or ASE}
Let $X$ and $Y$ be Banach spaces. If $T\in \ELL(X, Y)$ and $\|Tx_0\|=\|T\|$ for some $x_0 \in S_X$ such that $Tx_0\in\str(\overline{T(B_X)}),$ then for every $\eps>0$, there exists $S\in \RSE(X, Y)$ such that $\|S-T\|<\eps$, $\|Sx_0\|=\|S\|$ and $S-T$ has rank one. If $x_0\in\str(B_X)$, then we can choose $S$ so that $S\in \ASE(X, Y).$ 
\end{prop}

Although many of the results on the classical BPBp will not have an immediate extension to the RSE or ASE settings, we see that this can be done easily when one of the involved spaces has enough strongly exposed points. More concretely, Proposition \ref{prop: aproximate NA by RSE or ASE} has the following straightforward consequence.

\begin{cor}\label{cor:strBPBp}
    Let $X$ and $Y$ be Banach spaces and $\mathcal{I}(X, Y)$ a subspace of $\ELL(X, Y)$ containing all rank one operators, and suppose that $\mathcal{I}(X, Y)$ has the BPBp.
    \begin{itemize}
        \item[i)] If $\str(B_X)=S_X$ then $\mathcal{I}(X, Y)$ has the BPBp-ASE.
        \item[ii)] If $\str(B_Y)=S_Y$ then $\mathcal{I}(X, Y)$ has the BPBp-RSE.
    \end{itemize}
\end{cor}

In particular, it follows that $\str(B_X)=S_X$ is a sufficient condition for property [P], which was proved in \cite[Proposition~3.11]{JMR2023}. 
Another consequence of Corollary \ref{cor:strBPBp} is that, when $\str(B_X)=S_X$, the BPBp, BPBp-RSE and BPBp-ASE are all equivalent. By Remark \ref{rem: unif convex implies strongly exposed}, this corollary can be applied when $X$ or $Y$ is uniformly convex. We shall point out a first consequence which provides an improvement of \cite[Theorem~3.1]{KL2014} and \cite[Corollary~2.3]{ABGM2013}.
\begin{theorem}
    Let $X$ be a uniformly convex Banach space. Then, the pair $(X, Y)$ has the BPBp-ASE for every Banach space $Y$. Moreover, this property is satisfied by any subspace $\mathcal M\subseteq \ELL(X, Y)$ containing all finite rank operators.
\end{theorem} 

\begin{example} 
    Let $X$ be a subspace of $\ell_p$ or $L_p(\mu)$ for $1<p<\infty.$ Then, $(X, Y)$ has the BPBp-ASE for every Banach space $Y$.
\end{example}

We conclude this section with a short remark on a stronger version of the BPBp-ASE.

\begin{rem}
    Following \cite{CDKKLM2020}, a Banach space $X$ is said to be uniformly micro-semitransitive if there exists a function $\beta\colon (0,2)\longrightarrow \R^+$ such that whenever $x, y\in S_X$ satisfy that $\|x-y\|<\beta(\eps)$, then there exists $T\in \ELL(X)$ with $\|T\|=1$ such that $Tx=y \quad \text{and} \quad \|T- \id\|<\eps$. Using Proposition \ref{prop: aproximate NA by RSE or ASE} above, we can easily show that if $X$ is uniformly micro-semitransitive (therefore, $X$ is uniformly smooth and uniformly convex at the same time - see \cite[Corollaries 2.13 and 2.14]{CDKKLM2020}), then $(X, Y)$ satisfies the following stronger version of the BPBp-ASE for every Banach space $Y$: for every $\eps >0$, there exists $0<\eta (\eps)<\eps$ such that for every $T\in\ELL(X, Y)$ with $\|T\|=1$ and every $x_0\in S_X$ such that $\|Tx_0\|>1-\eta(\eps)$, there exists $S\in \ASE(X, Y)$ such that $\|S\|=\|Sx_0\|=1 \quad \text{and} \quad \|S-T\|<\eps$. This can be viewed as a pointwise BPBp-ASE in the sense that the new operator $S$ attains its norm at the same point where the operator $T$ almost attains its norm.
\end{rem}

\section{Characterisations of uniform convexity and complex uniform convexity}\label{sec: characterisations}
From now on, unless otherwise stated all the Banach spaces considered will be assumed to have dimension greater than or equal to two. The aim of this section is to show that the BPBp-RSE of the pair $(X, Y)$ characterises the uniform convexity of $Y$ when $X=L_1(\mu)$, where $\mu$ is any $\sigma$-finite measure, and the $\C$-uniform convexity of $Y$  when $X=L_\infty(\mu)$, where $\mu$ is any measure, or when $X$ is a $c_0$-sum of copies of a uniformly convex Banach space. In the real case, the BPBp-RSE also provides a characterisation of uniform convexity of $Y$ for these spaces.

These kinds of spaces have the property that they can be decomposed into $L$- or $M$-summands, and this fact will be used to provide the ``only if" part of the desired characterisations. Recall that if a Banach space can be decomposed as $X=X_1\oplus_1 X_2$ then we say that $X_1$ and $X_2$ are $L$-summands of $X$. If $X=X_1\oplus_\infty X_2$, then we say that $X_1$ and $X_2$ are $M$-summands of $X$. 

The key to our arguments lies in the following geometrical idea due to Lindenstrauss \cite{L1963}: if an operator $S\in \ELL(X, Y)$ attains its norm at a point $x$ in the interior of a segment $[x_1, x_2]\subseteq S_X$ and $Sx$ is an extreme point of $B_Y$, then $S$ must carry the whole segment to the point $Sx$. Actually, it suffices that $Sx$ is an extreme point of $S(B_X),$ which happens in particular when $S\in \RSE(X, Y)$ by Lemma \ref{lem:str-exp}. This is exactly the idea captured in Lemma \ref{lemma:easy-lemma-to-kill-propertyRSE-B} above, and it will be used to obtain necessary conditions on $Y$ for a pair $(X, Y)$ to have the BPBp-RSE when $X$ has an $L$- or an $M$-summand. 

\subsection{When the domain has an $L$-summand}\label{sec:char L1}
The proof of the following result can be easily adapted from \cite[Lemma~3.2]{ACKLM2015} using Lemma \ref{lemma:easy-lemma-to-kill-propertyRSE-B} instead of strict convexity of $Y$.
\begin{prop}\label{prop:l1sum}
    Let $X$ be a Banach space such that $X=X_1\oplus_1 X_2$ for non-trivial subspaces $X_1, X_2$ and let $Y$ be any Banach space. If $(X, Y)$ has the BPBp-RSE then $Y$ is uniformly convex.
\end{prop}

Turning to the particular case in which $X=L_1(\mu)$, we are now able to provide the first characterisation.

\begin{theorem}\label{teo:L_1 char}
    Let $Y$ be a Banach space, $(\Omega, \Sigma, \mu)$ any $\sigma$-finite measure space and $m\in \N, m\geq 2$. The following are equivalent.
   \begin{itemize}
       \item[i)] The pair $(\ell_1^m, Y)$ has the BPBp-RSE.
       \item[ii)] The pair $(\ell_1, Y)$ has the BPBP-RSE.
       \item[iii)] The pair $(L_1(\mu), Y)$ has the BPBp-RSE.
       \item[iv)] $Y$ is uniformly convex.
   \end{itemize}
\end{theorem} 
\begin{proof}
    Each of the statements i)-iii) imply iv) by Proposition \ref{prop:l1sum}. On the other hand, if $Y$ is uniformly convex then $(L_1(\mu), Y)$ has the BPBp for any $\sigma$-finite measure $\mu$ by Examples \ref{ex:BPBp}, so we can apply Corollary \ref{cor:strBPBp} to conclude that the pairs also have the BPBp-RSE.
\end{proof}

\begin{example}\label{ex: ell_1^2}
    The pair $(\ell_1^2, \ell_1^2)$ does not have the BPBp-RSE, although every pair of finite dimensional Banach spaces satisfy BPBp by \cite[Proposition~2.4]{AAGM2008} and $\ASE(\ell_1^2,\ell_1^2)$ is dense in $\ELL(\ell_1^2,\ell_1^2)=\NA(\ell_1^2,\ell_1^2)$ by Bourgain's result \cite{Bourgain1970}. This example also shows that $\C$-uniform convexity of the domain or the range is not enough to pass from the BPBp to the BPBp-RSE.
\end{example}

\subsection{When the domain has an $M$-summand in the complex case}\label{sec: M summand complex}
The parallel notion that emerges here is $\C$-uniform convexity. The result we present for domain spaces with $M$-summands improves \cite[Theorem~8]{KLL2016}, stated for $X=c_0, \ell_\infty^m$ and a $\C$-strictly convex Banach space $Y$. The proof is inspired by the argument in \cite[Lemma 3.2]{ACKLM2015}. 

\begin{theorem}\label{prop:M summand}
    Let $X, Y$ be complex Banach spaces such that $X=X_1\oplus_\infty X_2$ for non-trivial subspaces $X_1, X_2.$ Suppose that one of each conditions hold.
    \begin{itemize}
        \item[i)] The pair $(X, Y)$ has that BPBp-RSE.
        \item[ii)] The pair $(X, Y)$ has that BPBp and $Y$ is $\C$-strictly convex.
    \end{itemize}
    Then, $Y$ is $\C$-uniformly convex.
\end{theorem}
\begin{proof}
    Let $0<\eps<1.$ We will show that if $0\leq \delta <\eta\left(\frac{\eps}{2}\right)$, then $w_c(\delta)\leq \eps,$ where $\eta\left(\frac{\eps}{2}\right)<\frac{\eps}{2}$ is taken from the definition of BPBp-RSE. Let $y_0\in S_Y$ and $y\in Y$ be such that $\|y_0+wy\|\leq 1+\delta$ for every $w\in\T.$ 

For $i=1, 2$ we choose $e_j\in S_{X_j}$ and $e_j\in S_{X^*}$ to be such that $e_j^*(e_j)=1$ and $e^*_j(X_i)=0$ for $i\neq j$. Now, define $T\in \ELL(X, Y)$ by
\begin{equation*}
        T(x_1, x_2)=e_1^*(x_1)y_0 + e_2^*(x_2)y_1,
\end{equation*}
that is, $T=\tilde{T} \circ P$ where $P \colon X\longrightarrow \ell_\infty^2$ is given by $P(x_1, x_2)=(e_1^*(x_1), e_2^*(x_2))$ and satisfies $\|P\|\leq 1$ and $\tilde{T}\colon \ell_\infty^2\longrightarrow Y$ is defined by 
$$\tilde{T}(1,0)=y_0 \ \text{ and } \ \tilde{T}(0,1)=y_1.$$
    Then, $\|T(e_1, 0)\|=\|y_0\|=1,$ so $\|T\|\geq\|T(e_1, 0)\|=1$, and
    $$\|T\|\leq \|P\|\|\tilde{T}\|\leq \|\tilde{T}\|= \sup_{w\in \T}\|y_0+wy\|\leq 1 +\delta <1+\eta\left(\frac{\eps}{2}\right),$$
    the equality being true because $\tilde{T}$ attains its norm at an extreme point of $B_{\ell_\infty^2}$. This implies that $\|T(e_1, 0)\|=1>\|T\|-\eta\left(\frac{\eps}{2}\right)$ and, therefore
    \begin{equation*}
        \left\| \frac{T}{\|T\|}(e_1, 0) \right\|=\frac{1}{\|T\|}>1-\frac{\eta\left(\frac{\eps}{2}\right)}{\|T\|}>1-\eta\left(\frac{\eps}{2}\right),
    \end{equation*}
so we can apply that $(X,Y)$ has the BPBp-RSE to get an operator $S\in\RSE(X, Y)$ and a point $x\in S_{X}$ such that $$\|Sx\|=1=\|S\|, \ \ \|x-(e_1, 0)\|<\frac{\eps}{2} \ \text{ and } \ \left\|S-\frac{T}{\|T\|}\right\|<\frac{\eps}{2}.$$
Form the last inequality and the fact that $\left\|\frac{T}{\|T\|}-T\right\|=\|T\|-1<\eta\left(\frac{\eps}{2}\right)<\frac{\eps}{2}$ we derive that $\|S-T\|<\eps.$  If we write $x=(x_1, x_2)\in S_{X_1\oplus_\infty X_2}$ we have $\|x_2\|<\frac{\eps}{2}<1$ and $\|x_1\|=1$. Now, if $0<\lambda<1-\|x_2\|$ then $\|x_2\pm \lambda w e_2\|\leq \|x_2\|+\lambda<1$ for every $w\in \T,$ so $(x_1, x_2\pm\lambda w e_2)\in S_X$  and therefore $S(0, e_2)=0$ by Lemma \ref{lemma:easy-lemma-to-kill-propertyRSE-B} (or using $\C$-strict convexity of $Y$ if we only assume that $(X, Y)$ has the BPBp). This implies that 
\begin{equation*}
    \|y\|=\|T(0, e_2)\|=\|T(0, e_2)-S(0, e_2)\|<\eps.
\end{equation*}
Taking the supremum over all $y_0\in S_Y, y\in Y$ satisfying that $\|y_0+wy\|\leq 1+\delta$ for every $w\in\T$ we finally get that 
\begin{equation*}
    w_c(\delta)\leq \eps. \qedhere 
\end{equation*}
\end{proof}

\subsubsection{When the domain is $C_0(L)$}

The following theorem improves \cite[Theorem~2.4]{A2016}, which states that $\C$-uniform convexity of $Y$ is a sufficient condition for the pair $(C_0(L), Y)$ to have the BPBp.

\begin{theorem}\label{teo:C_0(L)}
    Let $L$ be a locally compact Hausdorff topological space and $Y$ be a complex Banach space. If $Y$ is $\C$-uniformly convex, then the pair $(C_0(L), Y)$ has the BPBp-RSE. If, moreover, $L$ is not connected, the converse also holds.
\end{theorem}
The condition that $Y$ is $\C$-uniformly convex is necessary when $L$ is not connected is a consequence of Theorem \ref{prop:M summand}, as in this case $C_0(L)$ has non-trivial $M$-summands. Indeed, $C_0(L)$ has an M-summand if and only if there is a subset $S\subset L$ which is open and closed in $L$ (see \cite[Example~1.4(a)]{M-ideals}). 
Let us point out the following interesting consequence of Theorem \ref{teo:C_0(L)}.

\begin{cor}\label{teo:L_infty char}
    Let $Y$ be a complex Banach space, $(\Omega, \Sigma, \mu)$ any measure space and $m\in \N, m\geq 2$. The following are equivalent.
   \begin{itemize}
       \item[i)] The pair $(\ell_\infty^m, Y)$ has the BPBp-RSE.
       \item[ii)] The pair $(c_0, Y)$ has the BPBp-RSE.
       \item[iii)] The pair $(\ell_\infty, Y)$ has the BPBP-RSE.
       \item[iv)] The pair $(L_\infty(\mu), Y)$ has the BPBp-RSE.
       \item[v)] $Y$ is $\C$-uniformly convex.
   \end{itemize}
\end{cor}

In order to prove Theorem \ref{teo:C_0(L)}, it only remains to show that $\C$-uniform convexity of $Y$ is sufficient using an argument based on the proof of \cite[Theorem~2.4]{A2016}.
For this, we will need some previous results. The first lemma is well known and its proof is straightforward.
\begin{lem}\label{lem:complex numbers}
    Let  $\lambda, \omega\in\C$ such that $|\lambda|, |\omega| \leq 1$ and $t \in (0,1)$. If $\re \omega \lambda >1-t,$ then $|\omega-\bar{\lambda}|<\sqrt{2t}.$
\end{lem}
In \cite[Theorem~B]{S1983} it was proved that every weakly compact operator whose domain is $C(K)$ can be approximated by norm-attaining weakly compact operators in the real case, and in \cite[Theorem~2]{ACKP1998} this result was generalized to $C_0(L)$ in both real and complex case. This result was recently improved in \cite[Theorem~3.15]{CDFJM2025} by showing that the approximation can be made using RSE operators. This will be a key step in the proof of Theorem \ref{teo:C_0(L)} and, in order to use it, we will take advantage of the following well-known fact whose proof can be found in \cite[Proposition~2.2]{A2016}, for example. 
\begin{prop}[\mbox{\cite[Proposition~2.2]{A2016}}]\label{prop: weakly compact C_0(L)}
    Let $Y$ be a $\C$-uniformly convex Banach space and $L$ any locally compact Hausdorff space. Then, every operator from $C_0(L)$ to $Y$ is weakly compact.
\end{prop}

Now, let us state some notation. Given a locally compact Hausdorff space $L$, $\mathcal{B}(L)$ will denote the space of Borel measurable and bounded complex valued functions defined on $L$, endowed with the supremum norm, which can be identified with a subspace of $C_0(L)^{**}$ by Riesz Theorem. If $B\subset L$ is a Borel measurable set, then we denote by $P_B\colon \mathcal{B}(L) \longrightarrow \mathcal{B}(L)$ the projection given by $P_B(f)=f\chi_B$ (we need to go to the bidual of $C_0(L)$ to consider the projection $f\chi_B$ of a function $f\in C_0(L)$, and then we can apply $T^{**}$ to this function). 

We are ready to prove that the condition in Theorem \ref{teo:C_0(L)} is sufficient for any locally compact Hausdorff space $L$.
\begin{proof}[Proof of Theorem \ref{teo:C_0(L)}]
    The purpose of this proof is to show that, in \cite[Theorem 2.4]{A2016}, the operator $S\in \NA(C_0(L), Y)$ and the point $f_3\in S_{C_0(L)}$ where it attains its norm can be chosen so that $Sf_3$ is a strongly exposed point of $\overline{S(B_{C_0(L)})}$, and therefore the pair satisfies the BPBp-RSE. We include the argument for the sake of completeness.   
    
    Let $0<\eps<1$ and take $\eta>0$ and $s>0$ as in \cite[Theorem 2.4]{A2016}. Let $T\in S_{\ELL(C_0(L), Y)}$ and $f_0\in S_{C_0(L)}$ be such that
\begin{equation*}
    \|Tf_0\|>1-s.
\end{equation*}
Choose $y_1^*\in S_{Y^*}$ such that
\begin{equation*}\label{eq:mu_1}
    \re y_1^*(Tf_0)=\|Tf_0\|>1-s.
\end{equation*}
By Riesz theorem we can identify $C_0(L)^*\equiv M(L),$ the space of Borel regular complex measures on $L$, and write $\mu_1=T^*y_1^*\in M(L).$ 

The idea of the proof is the following. We apply the proof of \cite[Theorem 2.4]{A2016} so that we can find an appropriate compact set $B\subseteq L$; then we apply \cite[Theorem~3.15]{CDFJM2025} to $C(B)$ to obtain an operator $S_2\in\RSE(C(B), Y)$ which is close to the restriction of $T$ to $C(B)$ (this restriction is defined using $T^{**}P_B$). Next, we extend $S_2$ to a new operator $S$ defined on $C_0(L)$ which attains its norm at some function $f_3$ such that $\|f_0-f_3\|<\eps$. Moreover, we show that $Sf_3\in \str(\overline{S(B_{C_0(L)})})$, so we can apply Proposition \ref{prop: aproximate NA by RSE or ASE} to approximate $S$ by an operator in $\RSE(C_0(L), Y)\cap S_{\ELL(C_0(L),Y)}$ which also attains its norm at $f_3,$ finishing the proof.

Let $g_1=\frac{d\mu_2}{d|\mu_2|}\in \mathcal B(L)$ and define $\tilde{S}\in\ELL(C_0(L), Y)$ by
\begin{equation*}
    \tilde{S}(f)=T^{**}(f\chi_B)+\eps_1 y_1^*(T^{**}(f\chi_B))T^{**}(\bar{g}_1\chi_B), f\in C_0(L),
\end{equation*}
where $\eps_1=\frac{1}{6}\frac{\delta(\frac{\eps}{9})}{1+\delta(\frac{\eps}{9})}$ and $B\subseteq L$ is an appropriate compact set obtained like in the proof of \cite[Theorem 2.4]{A2016} and satisfying that the restriction $g_1|_B$ is continuous and 
\begin{equation}\label{eq: f-g_1}
    \|(f_0-\bar{g}_1)\chi_B\|_\infty \leq \frac{\eps}{12}.
\end{equation}

We have that $\tilde{S}^{**}=\tilde{S}^{**}P_B$ and it can be shown as in \cite[Theorem 2.4]{A2016} that
\begin{equation*}\label{eq:estimate tilde S}
    1\leq 1-\eta+\eps_1(1-\eta)^2\leq \|\tilde{S}\|\leq 1+\eps_1,
\end{equation*}
so $|1-\|\tilde{S}\||\leq \eps_1.$ Now we extend $\tilde{S}$ to $C(B)\subseteq \mathcal{B}(L)$ using the biadjoint. For every $h\in C(B)$ we denote by $h\chi_B\in \mathcal{B}(L)$ the natural extension to $L$. Let $S_1\in \ELL(C(B), Y)$ be the operator given by
\begin{equation*}
    S_1(h)=\tilde{S}^{**}(h\chi_B) \in Y,
\end{equation*}
which is weakly compact by Proposition \ref{prop: weakly compact C_0(L)} and satisfies $\|S_1\|=\|\tilde{S}\|$ as $\tilde{S}^{**}=\tilde{S}^{**}P_B.$ Applying \cite[Theorem~3.15]{CDFJM2025}, there exists an operator $S_2\in\RSE(C(B),Y)$ and $h_1\in S_{C(B)}$ such that
\begin{equation*}\label{eq:S_2}
    \|\tilde{S}\|=\|S_2\|=\|S_2(h_1)\| \ \text{ and } \ \|S_2-S_1\|<\frac{\eps\eta}{2}.
\end{equation*}
By Lemma \ref{lem:str-exp}, $S_2(h_1)\in\str(\overline{S_2(B_{C(B)})})$. Choose $y_2^*\in S_{Y^*}$ such that
\begin{equation*}\label{eq:y_2}
    y_2^*(S_2(h_1))=\|S_2\|.
\end{equation*}
By rotating $h_1$ and $y_2^*$ if needed we can assume that $y_1^*(T^{**}(h_1\chi_B))\in \R_0^+.$

Write $R_2=\frac{S_2}{\|S_2\|}\in\RSE(C(B),Y)$ and $\mu_2=R_2^*y_2^*\in M(B).$ We have that $R_2(h_1)\in \str(\overline{R_2(B_{C(B)})})$ and 
\begin{equation}\label{eq:y_2 R_2}
    y_2^*(R_2(h_1))=\mu_2(h_1)=1=\|\mu_2\|=|\mu_2|(B).
\end{equation} 
By the proof of \cite[Theorem 2.4]{A2016} we also have that 
\begin{equation}\label{eq:y_2R_2g_1}
\begin{split}
    \re y_2^*(R_2(\bar{g_1}|_B))>1-6\eta -2\frac{\eta}{\eps_1}-\eps\eta.
\end{split}
\end{equation}

Let $g_2=\frac{d\mu_2}{d|\mu_2|}$, and we can assume that $|g_2|=1.$ Consider the measurable set
\begin{equation*}
    C=\{t\in B: \re(\bar{g_1}(t)+h_1(t))g_2(t)>2-\beta\},
\end{equation*}
where $\beta=\frac{\eps^2}{2\cdot 12^2}.$ Using \eqref{eq:y_2 R_2} and \eqref{eq:y_2R_2g_1} it can be shown that
\begin{equation*}\label{eq: mu B-C}
    |\mu_2|(B\backslash C)\leq \frac{6\eta + 2\frac{\eta}{\eps_1}+\eps\eta}{\beta}.
\end{equation*}
On the other hand, by Lemma \ref{lem:complex numbers} we also have
\begin{equation*}
    \|(g_1-g_2)\chi_C\|_\infty \leq \sqrt{\beta}=\frac{\eps}{12} \text{ and } \|(h_1-\bar{g_2})\chi_C\|_\infty \leq \sqrt{2\beta}=\frac{\eps}{12}.
\end{equation*}
This together with \eqref{eq: f-g_1} yields
\begin{equation*}\label{eq:h_1-f_0}
\begin{split}
    \|( &h_1-f_0)\chi_C\|_\infty
    \leq \frac{\eps}{4}.
\end{split}
\end{equation*}

By the inner regularity of $\mu_2$ there is a compact set $K_1\subset C$ such that
\begin{equation*}\label{eq:mu C-K}
    |\mu_2|(C\backslash K_1)<\frac{\eta\eps}{2}.
\end{equation*}
By the proof of \cite[Theorem 2.4]{A2016}, $K_1\neq \emptyset$ and $\|R_2^{**}(P_B-P_{K_1})\|\leq \frac{\eps}{9}.$

Denote by $T_2\in\ELL(C_0(L),Y)$ the operator defined by
\begin{equation*}
    T_2(f)=R_2(f|_B), \ f\in C_0(L).
\end{equation*}
It is satisfied that $\|T_2^{**}(I-P_{K_1})\|=\|R_2^{**}(P_B-P_{K_1})\|$ and $T_2^{**}(P_B-P_{K_1})=T_2^{**}(I-P_{K_1})P_B$ so
\begin{equation*}
    \|T_2^{**}(P_B-P_{K_1})\|\leq \frac{\eps}{9}.
\end{equation*}
Also define $R:C(B)\longrightarrow Y$ by $R(f)=T^{**}(f\chi_B)$ for every $f\in C(B).$ We have
\begin{equation*}
    \|(T_2^{**}-T^{**})P_B\|=\|R_2-R\|
\end{equation*}
By the definition of $S_1$ we have that 
\begin{equation*}
    \|S_1-R\|\leq \eps_1.
\end{equation*}

Now, we are going to use $R_2$ and $h_1$ to
define an operator $S\in \ELL(C_0(L), Y)$ and a point $f_3\in C_0(L)$ which satisfy the conditions of the BPBp, and we will show that $S(f_3)\in\str(\overline{S(B_{C_0(L)})})$, so the BPBp-RSE is also satisfied.

As $K_1\neq \emptyset,$ there exists $t_0\in K_1\subset C$ and $|h_1(t_0)|>1-\beta >1-\frac{\eps}{2}$, and there is an open set $V\subset B$ such that $t_0\in V\subseteq \{t\in B : |h_1(t)|>1-\frac{\eps}{2}\}$ and a function $v\in C(B)$ satisfying $v(B)\subseteq[0,1], v(t_0)=1$ and $\supp v\subseteq V.$ So we can define $h_i\in C(B)$, $i=2,3$ by
\begin{equation*}
    h_2(t)=h_1(t)+v(t)(1-|h_1(t)|)\frac{h_1(t)}{|h_1(t)|}, \ t\in B
\end{equation*}
and
\begin{equation*}
    h_3(t)=h_1(t)-v(t)(1-|h_1(t)|)\frac{h_1(t)}{|h_1(t)|}, \ t\in B.
\end{equation*}
Clearly $h_i\in B_{C(B)}$ for $i=2,3$ and $h_1=\frac{1}{2}(h_2+h_3)$. As $R_2\in \RSE(C(B), Y)$ and $\|R_2(h_1)\|=1=\|R_2\|,$ by Lemma \ref{lemma:easy-lemma-to-kill-propertyRSE-B} 
\begin{equation}\label{eq:R_2(h_2)}
R_2(h_2)=R_2(h_3)=R_2(h_1)\in\str(\overline{R_2(B_{C(B)}}).
\end{equation}
Also we have that $|h_2(t_0)|=1.$ As $\supp v\subseteq V\subseteq \{t\in B: |h_1(t)|<1-\frac{\eps}{2} \},$ for $t\in V$ we have
\begin{equation*}
    |h_2(t)-h_1(t)|\leq 1-|h_1(t)|<\frac{\eps}{2}.
\end{equation*}
For $t\in B\backslash V, h_2(t)=h_1(t)$ so $\|h_2-h_1\|<\frac{\eps}{2}$ and we obtain
\begin{equation*}
\begin{split}
    \|h_2-f_0|_C\|\leq \frac{3\eps}{4}.
\end{split}
\end{equation*}

As $B\subseteq L$ is compact, there exists a function $f_2\in C_0(L)$ that extends $h_2$ to $L$ (see, for example, \cite[Corollary 9.15, Theorem 12.4]{Jameson} and \cite[Theorems 17 and 18]{Kelley}). Considering the function $\Phi: \C\longrightarrow \C$ given by $\Phi(z)=z$ if $|z|\leq 1$ and $\Phi(z)=\frac{z}{|z|}$ if $|z|>1,$ which is continuous, by using $\Phi\circ f_2$ instead of $f_2$ if necessary we can assume that $f_2\in S_{C_0(L)}$ (as $h_2\in S_{C(B)}$, $\Phi\circ f_2$ is still an extension of $h_2$). There is an open set $G\subset L$ such that $K_1\subset G$ and 
\begin{equation}\label{eq:f2-f0}
    \|(f_2-f_0)\chi_G\|_\infty <\frac{7\eps}{8}.
\end{equation}
By Urysohn's lemma, there exists $u\in C_0(L)$ such that $u(L)\subset [0,1], u|_{K_1}=1$ and $\supp u \subseteq G$. Define $f_3\in B_{C_0(L)}$ by
\begin{equation*}
    f_3=uf_2+(1-u)f_0 
\end{equation*}
Observe that
\begin{equation}\label{eq: f_3=h_2}
    f_3(t)=f_2(t)=h_2(t) \ \forall t\in K_1, \ f_3(t)=f_0(t) \ \forall t\in L\backslash G
\end{equation}
and
\begin{equation*}
    |f_3(t)-f_0(t)|=u(t)|f_2(t)-f_0(t)|, \ \forall t\in G\backslash K_1.
\end{equation*}
Therefore, using \eqref{eq:f2-f0}
\begin{equation*}
    \|f_3-f_0\|<\eps.
\end{equation*}

Write $\lambda_0=\overline{h_2(t_0)}\in \T.$ Define the operator $S\in \ELL(C_0(L), Y)$ by
\begin{equation*}
    S(f)=R_2^{**}((f\chi_{K_1})|_B)+\lambda_0f(t_0)R_2^{**}(h_2\chi_{B\backslash K_1}), f\in C_0(L).
\end{equation*}
$S$ is well defined as $R_2$ is weakly compact. For every $f\in B_{C_0(L)}$ we have that $|\lambda_0f(t_0)|\leq 1$ and therefore
\begin{equation*}
\|(f\chi_{K_1})|_B+\lambda_0f(t_0)h_2\chi_{B\backslash K_1}\|_\infty \leq 1.
\end{equation*}
This implies that $S(B_{C_0(L)})\subseteq R_2^{**}(B_{\mathcal{B}(B)})\subseteq \overline{R_2(B_{C(B)})}$ (the second inclusion being a consequence of the weak compactness of $R_2$), so $\|S\|\leq 1$ and $$\overline{S(B_{C_0(L)}) }\subseteq\overline{R_2(B_{C(B)})}.$$

It is also satisfied that
\begin{equation*}
\begin{split}
    S(f_3)&=R_2^{**}((f_3\chi_{K_1})|_B)+\lambda_0f_3(t_0)R_2^{**}(h_2\chi_{B\backslash K_1})\\
    &=R_2^{**}(h_2) \ [\text{by \eqref{eq: f_3=h_2}}]\\
    &=R_2(h_2).
\end{split}
\end{equation*}
Therefore, by \eqref{eq:R_2(h_2)} we have that $\|S(f_3)\|=\|R_2(h_2)\|=1=\|S\|.$ Moreover, $S(f_3)=R_2(h_2)\in \str(\overline{S(B_{C_0(L)})})$ as it is a strongly exposed point of a set which contains  $\overline{S(B_{C_0(L)})}$. Applying Proposition \ref{prop: aproximate NA by RSE or ASE}, $S$ can be approximated by an operator in $\RSE(C_0(L), Y)$ with norm 1 which attains its norm at $f_3$. As we can show applying the proof of \cite[Theorem 2.4]{A2016} that $\|S-T\|<\eps$, we have finished. 
\end{proof}

\subsubsection{When the domain is $c_0(X)$}
In particular, the previous result applies when the domain is $c_0$ or $\ell_\infty^m.$ This can be extended to any $c_0$-sum and to any finite $\ell_\infty$-sum of copies of a uniformly convex Banach space, improving \cite[Theorem~2.4]{CK2022}.

\begin{theorem}\label{teo:c_0 sum}
    Let $X, Y$ be complex Banach spaces such that $X$ is uniformly convex. The following are equivalent.
    \begin{itemize}
        \item[i)] The pair $(c_0(X), Y)$ has the BPBp-RSE.
        \item[ii)] The pair $(\ell_\infty^m(X), Y)$ has the BPBp-RSE.
        \item[iii)] $Y$ is $\C$-uniformly convex.
    \end{itemize}
\end{theorem}

The proof is based on a modification of the argument in \cite[Theorem~2.4]{CK2022}. We will use the following result whose proof can be obtained from \cite[Lemma 2.3]{A2016}.

\begin{lem}[\mbox{\cite[Lemma~2.2]{CK2022}}]\label{lem:TP_A}
    Let $X$ be a Banach space and $Y$ be a $\C$-uniformly convex Banach space with modulus of $\C$-convexity $\delta_\C$. For a fixed $\eps>0$, if $T\in B_{\ELL(c_0(X), Y)}$ and $A\subseteq \N$ satisfy that $\|TP_A\| \geq 1-\frac{\delta_\C(\eps)}{1+\delta_\C(\eps)},$ then $\|T(I-P_A)\|\leq \eps$ where $P_A: c_0(X)\longrightarrow \ell_\infty^A(X)\subset c_0(X)$ is a projection on the components in $A.$ Analogously, if $T\in S_{\ELL(\ell_\infty^n(X), Y)}$ and $A\subset\{1, \dots, n\}$ satisfy that $\|TP_A\|\geq 1-\frac{\delta_\C(\eps)}{1+\delta_\C(\eps)},$ then $\|T(I-P_A)\|\leq \eps.$ 
\end{lem}

This result will also be used frequently.
\begin{lem}[\mbox{\cite[Lemma~2.3]{CK2022}}]\label{lem:convex series to T^*y^*}
    Let $X$ and $Y$ be Banach spaces and $0<\eta<1$ be given. Assume that $T\in S_{\ELL(c_0(X), Y)}, y_0^*\in S_{Y^*}$ and $x_0\in S_{c_0(X)}$ satisfy that
    \begin{equation*}
        y_0^*(Tx_0)=\|Tx_0\|>1-\eta.
    \end{equation*}
    Then, for $0<\eta'<1,$ we have
    \begin{equation*}
        \sum_{i \in A}\|(T^*y_0^*)(i)\|>1-\frac{\eta}{\eta'}.
    \end{equation*}
    where $A=\{ i\in \N : \re [(T^*y_0^*)(i)](x_0(i))>(1-\eta')\|(T^*y_0^*)(i)\|\}.$
    In particular, 
    \begin{equation*}
        \re \sum_{i\in A} [T^*y_0^*(i)](x_0(i))>\left( 1-\frac{\eta}{\eta'}\right)(1-\eta').
    \end{equation*}
\end{lem}

Finally, we will need the following geometrical observation for finite $\ell_\infty$-sums  that can be easily proved.
\begin{lem}\label{lem:strongly exposed in ell_infty}
    Let $X$ be a uniformly convex Banach space and $\ell_\infty^m(X)$ the $\ell_\infty$-sum of $m$ copies of $X$ for some $m\in \N.$ Then,
    \begin{equation*}
        \str (B_{\ell_\infty^m(X)})=\{x\in S_{\ell_\infty^m(X)}  : \|x(i)\|=1 \ \forall i=1, \dots, m\}.
        \end{equation*}
\end{lem}

\begin{proof}[Proof of Theorem \ref{teo:c_0 sum}]
The fact that i) or ii) imply iii) follows from Theorem \ref{prop:M summand}. We will show that iii) implies i).
We proceed as in the proof of \cite[Theorem 2.4]{CK2022}. Fix $0<\eps<1$ and set $\eta(\eps):=\min\{\frac{\eps}{16},\frac{\delta_\C(\frac{\eps}{16})}{1+\delta_\C(\frac{\eps}{16})} , \delta_X(\frac{\eps}{2})\}$ where $\delta_X$ is the modulus of convexity of $X$ and $\delta_\C$ is the modulus of $\C$-convexity of $Y.$ Let $T\in S_{\ELL(c_0(X),Y)}$ and $x_0\in S_{c_0(X)}$ such that
\begin{equation*}
    \|Tx_0\|>1-\frac{\eta^6}{64}.
\end{equation*}
Observe that we can assume without loss of generality that $x_0$ has finite support as such elements are dense in $c_0(X).$ Choose $y_0^*\in S_{Y^*}$ satisfying $y_0^*(Tx_0)=\|Tx_0\|>1-\frac{\eta^6}{64}$ and define
\begin{equation*}
    A:=\left\{i\in \N:\re [(T^*y_0^*)(i)](x_0(i))>\left(1-\frac{\eta^3}{8}\right)\|(T^*y_0^*)(i)\|\right\}.
\end{equation*}
Applying Lemma \ref{lem:convex series to T^*y^*} with $\eta'=\frac{\eta^3}{8}$ yields
\begin{equation*}
    \sum_{i\in A}\|(T^*y_0^*)(i)\|>1-\frac{\eta^3}{8},
\end{equation*}
so $\|TP_A\|> 1-\frac{\eta^3}{8}\geq 1-\frac{\delta_\C(\frac{\eps}{16})}{1+\delta_\C(\frac{\eps}{16})}$, and by Lemma \ref{lem:TP_A} 
\begin{equation*}
    \|TP_A-T\|<\frac{\eps}{16}.
\end{equation*}
Write $\hat{T}\in B_{\ELL(\ell_\infty^A(X), Y)}$ for the canonical restriction of $T$ and $\hat{x}_0=(\hat{x}_0(i))_{i\in A}=\left(\frac{x_0(i)}{\|x_0(i)\|}\right)_{i\in A}\in S_{\ell_\infty^A(X)}.$ Observe that $\hat{x}_0\in \str(B_{\ell_\infty ^A(X)})$ by Lemma \ref{lem:strongly exposed in ell_infty}. Then,
\begin{equation*}
    \|\hat{T}\hat{x}_0\|>1-\frac{\eta^3}{4}
\end{equation*}
and $\|\hat{x}_0(i)-x_0(i)\|<\frac{\eta^3}{8}$ for $i\in A$ (by the definition of $A$). This means that $P_A(x_0)$ is close to a strongly exposed point of $B_{\ell_\infty ^A(X)}$.

Choose $y_1^*\in S_{Y^*}$ such that $\re y_1^*(\hat{T}\hat{x}_0)=\|\hat{T}\hat{x}_0\|$ and define $R\in \ELL(\ell_\infty^A(X), Y)$ by
\begin{equation*}
    Rz:=\hat{T}z+\eta y_1^*(\hat{T}z)\frac{\hat{T}\hat{x}_0}{\|\hat{T}\hat{x}_0\|}, \ z\in \ell_\infty^A.
\end{equation*}
By Bourgain's result in \cite{Bourgain1970}, as $\ell_\infty^A(X)$ has the RNP, there exists $Q\in \ASE(\ell_\infty^A(X), Y)$ such that $Q$ attains its norm at $w_0\in S_{\ell_\infty^A(X)}, \|Q\|=\|R\|$ and $\|Q-R\|<\frac{\eta^3}{4}.$ By Lemma \ref{lem:str-exp}, $w_0\in \str(B_{\ell_\infty^A(X)}).$ 

We have
\begin{equation*}
\begin{split}
    1-\frac{\eta^3}{4}+\eta\left(1-\frac{\eta^3}{4}\right)\leq\|R\hat{x}_0\|\leq \frac{\eta^3}{4}+1+\eta|y_1^*(\hat{T}w_0)|
\end{split}
\end{equation*}
Rotating $w_0$ if necessary, we may assume that $|y_1^*(\hat{T}w_0)|=\re y_1^*(\hat{T}w_0)$ and, from above,
\begin{equation*}
    \re y_1^*\left(\hat{T}\left(\frac{w_0+\hat{x}_0}{2}\right)\right)\geq 1-\eta^2.
\end{equation*}
Define
\begin{equation*}
    B:=\left\{i\in A : \re \left[\hat{T}^*y_1^*(i)\right]\left(\frac{w_0+\hat{x}_0}{2}(i)\right)>(1-\eta)\|\hat{T}^*y_1^*(i)\|\right\}.
\end{equation*}
Arguing like in \cite[Theorem 2.4]{CK2022} we can show that
\begin{equation*}
    \|\hat{T}(I-P_B)\|<\frac{\eps}{16}
\end{equation*}
and 
\begin{equation*}
    \|w_0(i)-\hat{x}_0(i)\|< \frac{\eps}{2},
\end{equation*}
using Lemma \ref{lem:convex series to T^*y^*} and uniform convexity of $Y$, respectively.

Now, define $\tilde{S}\in \ELL(\ell_\infty^A(X), Y)$ by $\tilde{S}=QP_B+Q(I-P_A)U,$ where $U\in B_{\ELL(\ell_\infty^A(X))}$ is chosen so that $U(E_i(\hat{x}_0(i))=E_i(w_0(i))$ for every $i\in A$, being $E_i: X\longrightarrow \ell_\infty^A(X)$ the $i$th injection map. Let $S$ be the canonical extension of $\frac{\tilde{S}}{\|\tilde{S}\|}$, that is, $S=\frac{\tilde{S}}{\|\tilde{S}\|}\circ P_A$. Define
\begin{equation*}
    z_0(i):=\left\{ \begin{array}{lcc} w_0(i) & \text{if} & i\in B \\\\ \hat{x}_0(i) & \text{if} & i\in A\backslash B\\ \\ x_0(i) & \text{otherwise} & 
    \end{array}\right.
\end{equation*}
Then, we can show arguing like in \cite[Theorem 2.4]{CK2022} that 

\begin{equation*}
   \|S\|=\|Sz_0\|=1, \quad \|S-T\|<\frac{11\eps}{16}<\eps \quad \text{and}\quad \|z_0-x_0\| <\eps.
\end{equation*}

Consider $\frac{\tilde{S}}{\|\tilde{S}\|}\in S_{\ELL(\ell_\infty^A(X), Y)}$. Then, 
$\left\|\frac{\tilde{S}}{\|\tilde{S}\|}(P_A(z_0))\right\|=\|Sz_0\|=1.$ We have that

\begin{equation*}
    P_A(z_0)=\left\{ \begin{array}{lcc} w_0(i) & \text{if} & i\in B \\\\ \hat{x}_0(i) & \text{if} & i\in A\backslash B\\ \\ 0 & \text{otherwise} & 
    \end{array}\right.
\end{equation*}
which is a strongly exposed point of $B_{\ell_\infty^A(X)}$ by Lemma \ref{lem:strongly exposed in ell_infty} (recall that $\|w_0(i)\|=1$ as $w_0\in \str(B_{\ell_\infty^A(X)})$ and $\|\hat{x}_0(i)\|=\|\frac{x_0(i)}{\|x_0(i)\|}\|=1$ for every $i\in A$). Hence, we can approximate $\frac{\tilde{S}}{\|\tilde{S}\|}$ by some $\tilde{V}\in \ASE(\ell_\infty^A(X), Y)$ such that
\begin{equation*}
    \|\tilde{V}P_A(z_0)\|=\|\tilde{V}\|=1
\end{equation*}
and $\left\|\tilde{V}-\frac{\tilde{S}}{\|\tilde{S}\|}\right\|<\frac{\eps}{16}.$ If we define $V=\tilde{V}P_A\in \ELL(c_0(X), Y)$ we have that
\begin{equation*}
    \|V\|\leq 1=\|Vz_0\|=\|\tilde{V}P_A(z_0)\|\leq 1,
\end{equation*}
so $\|V\|=1=\|Vz_0\|$, and $$\|V-T\|\leq \|\tilde{V}P_A-\frac{\tilde{S}}{\|\tilde{S}\|}P_A\|+\|S-T\|
\leq\left\|\tilde{V}-\frac{\tilde{S}}{\|\tilde{S}\|}\right\|+\|S-T\|<\frac{12\eps}{16}<\eps.$$
We claim that $V\in\RSE(c_0(X), Y)$ at $z_0.$ Indeed, let $(x_n)\subseteq B_{c_0(X)}$ be a sequence such that $\|Vx_n\|=\|\tilde{V}P_A(x_n)\|\longrightarrow \|V\|=1$. Then, as $\tilde{V}\in \ASE(\ell_\infty^A(X), Y)$, there exists a sequence $(\theta_n)\subseteq \T$ such that $\theta_nP_A(x_n)\longrightarrow Pz_0,$ so
\begin{equation*}
    \theta_n Vx_n\longrightarrow Vz_0,
\end{equation*}
as desired.

Finally, to see that iii) implies ii) just observe that the same argument (with the same constants) can be applied to the pair $(\ell_\infty^A(X), Y)$ considering the canonical extension to $c_0(X)$ of each $T: \ell_\infty^A(X) \longrightarrow Y.$
\end{proof}

\subsection{When the domain has an $M$-summand in the real case}\label{sec: M summand real}
In the real case, we can obtain an analogous result to Proposition \ref{prop:l1sum} when the domain has an M-summand. This extends some already known results for the BPBp which apply when $Y$ is strictly convex and $X=c_0, \ell_\infty^m$ \cite[Theorem~2.7]{K2013} or $X=\left[\bigoplus_{k=1}^\infty \ell_2^k\right]_{c_0}$ \cite[Theorem~2.3]{GL2025}. Notice that we state the result in the real case as Example \ref{ex: c0 ell_1} shows that in the complex case it is not true.

\begin{theorem}\label{prop:l infty sum real}
    Let $X, Y$ be real Banach spaces such that $X=X_1\oplus_\infty X_2$ for non-trivial subspaces $X_1, X_2.$ Suppose that one of these conditions holds.
    \begin{itemize}
        \item[i)] The pair $(X, Y)$ satisfies the BPBp-RSE.
        \item[ii)] The pair $(X, Y)$ has the BPBp and $Y$ is strictly convex.
    \end{itemize}
    Then, $Y$ is uniformly convex.
\end{theorem}
\begin{proof}
    The argument is based on the proof of \cite[Theorem~2.7]{K2013}. Suppose that $(X, Y)$ satisfies the BPBp-RSE and that $Y$ is not uniformly convex. Then, there exist $0<\eps<1$ and $(x_n)_{n\in\N}, (y_n)_{n\in\N} \subseteq S_Y$ such that $$\left\|\frac{x_n+y_n}{2}\right\|\longrightarrow 1 \quad\text{and}\quad \|x_n-y_n\|>\eps$$ for every $n\in \N.$ Choose $e_1\in S_{X_1}, e_2\in S_{X_2}, e_1^*,e_2^*\in S_{X^*}$ such that $e_i^*(e_j)=\delta_{ij}$ for $i,j=1,2$ and define $T_n\colon X_1\oplus_\infty X_2\longrightarrow Y$ by
    \begin{equation*}
        T_n(x)=e_1^*(x)\frac{x_n+y_n}{2}+e_2^*(x)\frac{x_n-y_n}{2} \quad (x\in X_1\oplus_\infty X_2)
    \end{equation*}
    for each $n\in \N.$ Then, $\|T_i\|\leq 1$ for each $i\in \N$ as in the real case we have
    \begin{equation}\label{eq: real case vs complex}
    \begin{split}
        \{\|T_i(x)\| : x\in B_X\} &\subseteq \left\{ \left\|t\frac{x_i+y_i}{2}x_i+s\frac{x_i-y_i}{2}y_i\right\| : -1\leq t,s\leq 1 \right\} \\
        &=\left\{ \left\|\frac{t+s}{2}x_i+\frac{t-s}{2}y_i\right\| : -1\leq t,s\leq 1 \right\} \\
        &=\{\|ax+by\|: |a|+|b|\leq 1\},
    \end{split}
    \end{equation}
    so $\sup_{x\in B_X} \|Tx\|\leq \max \{\|ax+by\|: |a|+|b|\leq 1\}\leq 1.$ Observe also that $\lim_{i\rightarrow \infty}\|T_i(e_1)\|=1.$ 
    
    Take $j\in \N$ such that $\|T_j(e_1)\|>1-\eta(\frac{\eps}{4})$, where $\eta(\frac{\eps}{4})<\frac{\eps}{4}$ is taken from the BPBp-RSE of the pair $(X, Y).$ In particular, $$\left\|\frac{T_j}{\|T_j\|}(e_1)\right\|>1-\eta\left(\frac{\eps}{4}\right) \quad\text{and}\quad 
   \left\|\frac{T_j}{\|T_j\|}-T_j\right\|<\frac{\eps}{4}.$$ Applying the BPBp-RSE, we obtain an operator $\tilde{T}\in \RSE(c_0, Y)$ and $x\in S_{c_0}$ such that
   \begin{equation*}
       \|\tilde{T}x\|=1=\|\tilde{T}\|, \quad \left\|\tilde{T}-\frac{T_j}{\|T_j\|}\right\|<\frac{\eps}{4} \quad\text{and}\quad \|x-e_1\|<\frac{\eps}{4}<1,
   \end{equation*}
   and so $\|\tilde{T}-T_j\|<\frac{\eps}{2}.$
   
   Write $x=x_1+x_2\in S_X$ with $x_1\in B_{X_1}, x_2\in B_{X_2}$, then
    \begin{equation*}
        \|x_2\|\leq \max\{ \|x_1-e_1\|, \|x_2\|\}=\|x_1-e_1+x_2\|=\|x-e_1\|<1,
    \end{equation*}
    so there is $0<\delta<1-\|x_2\|$ such that $\|x_2\pm \delta e_2\|\leq \|x_2\|+\delta<1$ and, therefore
    \begin{equation*}
        \|x\pm\delta e_2\|=\max\{\|x_1\|,\|x_2\pm \delta e_2\|\}\leq 1.
    \end{equation*}
    If $\tilde{T}$ attains its norm at $x,$ applying that $\tilde{T}\in \RSE(X, Y)$ together with Lemma \ref{lemma:easy-lemma-to-kill-propertyRSE-B} yields that $\tilde{T}e_2=0$, so $\tilde{T}(e_1\pm e_2)=\tilde{T}(e_1).$ But then,
    \begin{equation*}
    \begin{split}
        \|x_j-y_j\|&=\|T_j(e_1+e_2)-T_j(e_1-e_2)\|\\
        & \leq \|T_j(e_1+e_2)-\tilde{T}(e_1+e_2)\|+\|\tilde{T}(e_1-e_2)+T_j(e_1-e_2)\|\\
        &<\frac{\eps}{2}+\frac{\eps}{2}=\eps,
    \end{split}
    \end{equation*}
    a contradiction. If we assume that the pair only satisfies the BPBp and $Y$ is stricly convex, then the equality $\tilde{T}(e_1\pm e_2)=\tilde{T}(e_1)$ follows from the strict convexity of $Y$ and the same argument works.
\end{proof}

Let us remark that the previous argument cannot be translated to the complex case, as the set equalities in equation \ref{eq: real case vs complex} are only valid in the real case. Indeed, it is essential for this argument that the operators $T_n$ have norm $\|T_n\|\leq 1.$

\begin{example}\label{ex: c0 ell_1}
    The pair $(c_0, \ell_1)$ does not satisfy the BPBp-RSE in the real case, although it does in the complex case in virtue of Theorem \ref{teo:L_infty char}. Moreover, the pair formed by the respective underlying real spaces $(c_0(\ell_2^2),\ell_1(\ell_2^2))$  does not satisfy this property applying Theorem \ref{prop:l infty sum real} and the fact that $\ell_1(\ell_2^2)$ is not uniformly convex. This shows that the BPBp-RSE of a pair is not necessarily inherited by the underlying real structure of the spaces.
\end{example}

Now we are able to give another characterisation of uniform convexity in the real case in terms of the BPBp-RSE.
\begin{theorem}\label{teo:c_0 real}
    Let $Y$ be a real Banach space. The following are equivalent.
    \begin{itemize}
        \item[i)] The pair $(\ell_\infty^m, Y)$ has the BPBp-RSE for some $m\in\N, m\geq 2$.
        \item[ii)] The pair $(c_0, Y)$ has the BPBp-RSE.
        \item[iii)] The pair $(\ell_\infty, Y)$ has the BPBp-RSE.
        \item[iv)] The pair $(L_\infty(\mu), Y)$ has the BPBp-RSE, where $\mu$ is any measure.
        \item[v)] The pair $(C(K), Y)$ has the BPBp-RSE, where $K$ is any not connected compact Hausdorff space.
        \item[vi)] $Y$ is uniformly convex.
    \end{itemize}
\end{theorem}
\begin{proof}
The fact that each item implies vi) follows from Theorem \ref{prop:l infty sum real}; vi) implies i), ii), iii), iv) and v) by Examples \ref{ex:BPBp} and Corollary \ref{cor:strBPBp}.
\end{proof}

Observe that the pair $(Z, Y)$, where $Z$ is a $c_0$-sum or a finite $\ell_\infty$-sum of copies of a uniformly convex Banach space and $Y$ is uniformly convex, has the BPBp in the real setting as the proof of \cite[Theorem 2.4]{CK2022} is still valid in this case, so we obtain the following.
\begin{theorem}\label{teo:c_0(X) real}
     In the real case, the pair $(c_0(X), Y)$, where $X$ is a uniformly convex Banach space, satisfies the BPBp-RSE if and only if $Y$ is uniformly convex. The same holds if the domain is $\ell_\infty^m(X).$
\end{theorem}

\section{The RSE Bollobás theorem for some classes of operators}\label{sec:subspaces of operators}
In this section, on one hand, we extend the results in Section \ref{sec: characterisations} to compact operators and, on the other, we study some conditions under which the ACK$_\rho$ structure of the range is enough to ensure the BPBp-RSE or the BPBp-ASE of some subspaces of operators. Recall that the BPBp-RSE for an operator subspace is defined immediately after Definition \ref{defi:BPBp-RSE}.

The main subspaces of $\ELL(X, Y)$ that we are going to consider in this section are the operator ideals of finite rank, compact, and weakly compact operators, which will be denoted  by $\FR(X, Y)$, $\KP(X, Y)$ and $\mathcal W(X, Y)$, respectively. Also recall that $T\in\ELL(X, Y)$ is a \emph{Dunford-Pettis operator} if it sends weakly convergent sequences to norm convergent sequences, and we will denote this operator ideal by $\DP(X, Y)$. We have that $$\FR\subseteq \KP\subseteq \DP.$$ 
An operator $T\in\ELL(X, Y)$ is called an \emph{Asplund operator} if it factors through an Asplund space. The class of all Asplund operators between two Banach spaces will be denoted by $\mathcal{A}(X, Y).$ Finally, we say that $T \in \Lin(X,Y)$ is a \textit{Radon-Nikod\'ym operator} (\emph{RN operator}, in short) if for every finite measure space $(\Omega, \Sigma, \mu)$ and every vector measure of bounded variation $G\colon\Sigma \to X$, the measure $T\circ G \colon\Sigma \to Y$ has a Bochner integrable Radon-Nikod\'ym derivative \cite{Edgar}. It is satisfied that
$$\FR\subseteq\KP\subseteq \mathcal W\subseteq \mathcal A \cap \mathcal{RN}.$$

\subsection{Compact operators from $L_1(\mu)$ or $C_0(L)$} \label{sec: char for compact} In this section 
we show that the results and characterisations in the previous section are also valid for compact operators, although we will state the results in the most general form we are able to. Specifically, we characterise the BPBp-RSE for compact operators whose domain is $L_1(\mu)$ or $L_\infty(\mu)$. Moreover, we extend some of the techniques in \cite{DGMM2018} for passing from sequence spaces to function spaces to the RSE setting and derive some consequences for compact operators defined on $C_0(L)$ and $\ell_1$ preduals. 

Let us begin by observing that Proposition \ref{prop:l1sum} and Theorems \ref{prop:M summand} and \ref{prop:l infty sum real} are actually valid for any subspace of operators containing all of rank-two. This follows simply from the fact that the operators used in the beginning of the proofs are of rank-two. We can apply this together with some known results in the literature to obtain the following natural consequences.

\begin{theorem}
    Let $Y$ be a Banach space and $(\Omega, \Sigma, \mu)$ any $\sigma$-finite measure space. The following are equivalent:
   \begin{itemize}
       \item[i)] $\FR(L_1(\mu),Y)$ has the BPBp-RSE
       \item[ii)] $\KP(L_1(\mu),Y)$ has the BPBp-RSE
       \item[iii)] $\mathcal W(L_1(\mu),Y)$ has the BPBp-RSE
       \item[iv)] $\mathcal{RN}(L_1(\mu),Y)$ has the BPBp-RSE.
       \item[v)] $\ELL(L_1(\mu), Y)$ has the BPBp-RSE.
       \item[vi)] $Y$ is uniformly convex.
   \end{itemize}
\end{theorem}
\begin{proof}
    The equivalence between v) and vi) is Theorem \ref{teo:L_1 char}. The fact that each item i)-iv) imply vi) follows from Proposition \ref{prop:l1sum}, as remarked before. Finally, vi) implies i)-iv) by \cite[Corollary~2.4, Remark~2.5]{ABGKM2014} and Corollary \ref{cor:strBPBp}.
\end{proof}

Also, we can derive the following:
\begin{theorem}\label{teo:C_0(L) compact real case}
    Let $Y$ be a real Banach space and $L$ a locally compact Hausdorff topological space. If $Y$ is uniformly convex, the pair $(C_0(L), Y)$ satisfies the BPBp-RSE for compact operators.
\end{theorem}

\begin{proof}
    If $Y$ is uniformly convex the subspace $\KP(C_0(L), Y)$ has the BPBp by \cite[Corollary~3.5]{DGMM2018}, and so it has the BPBp-RSE by Corollary \ref{cor:strBPBp}. The converse implication holds using Theorem \ref{prop:l infty sum real}.
\end{proof}

Therefore, applying Theorem \ref{prop:l infty sum real} we get a characterisation for compact operators from $L_\infty(\mu).$
\begin{cor}\label{cor:L_infty compact real case}
    Let $Y$ be a real Banach space and $(\Omega, \Sigma, \mu)$ a measure space. Then, the subspace $\KP(L_\infty(\mu),Y)$ has the BPBp if and only if $Y$ is uniformly convex.
\end{cor}

In the complex case, with a little more effort we can obtain some similar results to Theorem \ref{teo:C_0(L) compact real case} and Corollary \ref{cor:L_infty compact real case} using $\C$-uniform convexity. In order to do so, we will adapt the abstract techniques introduced in \cite{DGMM2018} to transfer the BPBp for compact operators from sequence spaces to function spaces to the RSE setting. Observe that they are valid for both real and complex cases.

\begin{lem}\label{lem:proyections-K}
    Let $X$ and $Y$ be Banach spaces. Suppose that there exists a function $\eta: \R^+\longrightarrow \R^+$ such that given $\delta>0, x_1^*, ..., x_n^*\in B_{X^*},$ and $x_0\in S_X,$ we can find a norm-one operator $P\in \ELL(X, X)$ and a norm-one operator $i\in\ELL(P(X), X)$ such that
    \begin{itemize}
        \item[i)] $\|P^*x_j^*-x_j^*\|<\delta$ for $j=1,...,n$;
        \item[ii)] $\|i(Px_0)-x_0\|<\delta$;
        \item[iii)] $P\circ i=\mathrm{Id}_{P(X)}$;
        \item[iv)] the pair $(P(X), Y)$ has the BPBp-RSE for compact operators with the function $\eta$. 
    \end{itemize}
    Then, the pair $(X, Y)$ has the BPBp-RSE for compact operators.
\end{lem}
\begin{proof}
    Just follow the proof of \cite[Lemma~2.1]{DGMM2018} and observe that, as $\|P\|=1,$ if $\tilde{S}\in \mathcal{K}(P(X), Y)$ belongs to RSE, then so does $S=\tilde{S}P$.
\end{proof}
In particular, this can be applied to $c_0$ considering the natural projections $P_m: c_0 \longrightarrow \ell_\infty^m, m\in\N$. This will be used together with some stability results which are valid for compact operators as well as for the general case and are analogous to the stability results in \cite{ACKLM2015}. It is worth mentioning that these results were extended to absolute summands in \cite{CDJM2019}. 
\begin{lem}\label{prop:stability-K}
    Let $\{X_i\colon  i\in I\}$ and $\{Y_j\colon j\in J\}$ be two families of Banach spaces and $X, Y$ their $c_0$-, $\ell_\infty$- or $\ell_1$-sum, respectively. If $(X, Y)$ has the BPBp-RSE (for compact operators) with $\eta(\eps)$ for $\eps\in (0,1)$ then so does the pair $(X_i, Y_j)$ for every $i\in I, j\in J.$
\end{lem}

We notice that the arguments to prove this Lemma are standard and are motivated by the proofs of \cite[Proposition 2.3, Proposition 2.6, Proposition 2.7]{ACKLM2015}. Our job here is to show that the involved operators are still RSE. We include a proof for the sake of completeness.
\begin{proof}
We state the proof for the general BPBp-RSE and we follow the argument of the stability results in \cite{ACKLM2015}. In the compact case, the same proof holds as compact operators between two Banach spaces are a linear subspace and they satisfy the so-called ideal property, that is, given arbitrary Banach spaces $X_0, Y_0$ we have $R\circ S\circ T \in \mathcal K(Z, W)$ for any $S\in S(X_0, Y_0), T\in \ELL(Z, X_0), R\in \ELL(Y_0, W)$, and for every Banach spaces $Z$ and $W.$ 

From now on, suppose that the pair $(X, Y)$ satisfies the BPBp-RSE. We separate the proof into three cases.

    \textit{Case 1.} $X=\left[\bigoplus_{i\in I} X_i\right]_{\ell_1}, Y=\left[\bigoplus_{j\in J} Y_j\right]_{\ell_\infty}$ or $Y=\left[\bigoplus_{j\in J} Y_j\right]_{c_0}$

    Let $E_i$ and $F_j$ the natural isometric embeddings of $X$ and $Y$,  respectively, and $P_i$ and $ Q_j$ the natural norm one projections from $X$ and $Y$ onto $X_i$ and $Y_j$, respectively. Fix $h\in I, k\in J$ and let $0<\eps<1$ and $T\in S_{\ELL(X_h, Y_k)}, x_h\in S_{X_h}$ such that $\|Tx_h\|>1-\eta(\eps).$ If we consider $\tilde{T}=F_kTP_h$ we have that $\|\tilde{T}\|=\|T\|=1$ and $\|\tilde{T}(E_hx_0)\|>1-\eta(\eps).$ Then, by the proof of \cite[Proposition~2.3]{ACKLM2015} if $\tilde{S}\in \ELL(X, Y)$ and $x_0\in S_X$ are such that 
    \begin{equation*}
        \|\tilde{S}x_0\|=1=\|\tilde{S}\|, \quad \|\tilde{T}-\tilde{S}\|<\eps \text{ and } \|x_0-E_h(x_h)\|<\eps
    \end{equation*}
    then the operator $S=Q_k\tilde{S}E_h\in \ELL(X_h, Y_h)$ and the point $P_h(x_0)\in B_X$ satisfy
    \begin{equation*}
        \|S(P_h(x_0))\|=1=\|S\|=\|P_h(x_0)\|, \text{ } \|S-T\|<\eps \text{ and } \|x_h-P_hx_0\|<\eps,
    \end{equation*}
    so it only remains to show that $S\in \RSE(X_h, Y_k)$ if $\tilde{S}\in \RSE(X, Y).$ Let $(x_n)\subseteq B_{X_h}$ such that $\|Sx_n\|= \|Q_k\tilde{S}(E_hx_n)\|\longrightarrow \|S\|.$ Then, 
    \begin{equation*}
        \|Q_k\tilde{S}(E_hx_n)\|\leq \sup_{j\in J} \|Q_j\tilde{S}E_h(x_n)\|=\|\tilde{S}E_h(x_n)\|\leq \|\tilde{S}\|=1,
    \end{equation*}
    so $\|\tilde{S}E_h(x_n)\|\longrightarrow 1=\|\tilde{S}\|.$ As $\tilde{S}\in \RSE(X, Y),$ there exists $(\lambda_n)\subseteq \T$ such that $\lambda_nE_h(x_n)\longrightarrow x_0.$ This means that
    \begin{equation*}
        \lambda_n x_n =P_h(E_h(\lambda_nx_n))\longrightarrow P_h x_0,
    \end{equation*}
    as we wanted.

    \textit{Case 2.} $X=\left[\bigoplus_{i\in I} X_i\right]_{c_0}$ or $X=\left[\bigoplus_{i\in I} X_i\right]_{\ell_\infty}$ and $Y$ is any Banach space. 

    Without loss of generality, assume that $X=X_1\oplus_\infty X_2$ and let $\eps>0, T:X_1\longrightarrow Y, \|T\|=1$ and $x_0\in S_{X_1}$ such that $\|Tx \|>1-\eta(\eps)$. Define $\tilde{T}: X\longrightarrow Y$ by 
    $$\tilde{T}(x_1,x_2)=Tx_1,$$
    then $\|\tilde{T}\|=1$, $\|\tilde{T}(x_0, 0)\|>1-\eta(\eps)$ and there exist $\tilde{S}\in\RSE(X, Y), (x_1',x_2')\in S_X$ such that 
    \begin{equation*}
        \|\tilde{S}\|=1=\|\tilde{S}(x_1',x_2')\|, \text{ } \|\tilde{T}-\tilde{S}\|<\eps \text{ and } \|(x_1', x_2')-(x_0,0)\|<\eps.
    \end{equation*}
    By the proof of \cite[Proposition~2.6]{ACKLM2015}, we have that the operator $S\in \ELL(X_1, Y)$ defined as $Sx_1=\tilde{S}(x_1, 0)$ satisfies
    \begin{equation*}
        \|S\|=1=\|S(x_1')\|, \quad \|T-S\|<\eps \text{ and } \|x_1'-x_0\|<\eps,
    \end{equation*}
    and we only need to show that $S\in \RSE(X, Y).$ Let $(x_n)\subseteq B_{X_1}$ be such that $\|Sx_n\|=\|\tilde{S}(x_n,0)\|\longrightarrow\|S\|.$ Then, there exists $(\lambda_n)\subseteq \T$ such that $\lambda_n\tilde{S}(x_n,0)\longrightarrow \tilde{S}(x_1',0),$ that is, $$\lambda_nSx_n\longrightarrow Sx_1'.$$
    
    \textit{Case 3.} $X$ is any Banach space and $Y=\left[\bigoplus_{j\in J} Y_j\right]_{\ell_1}$

    Without loss of generality, we can assume that $Y=Y_1\oplus_1 Y_2$. Let $\eps>0, T\in \ELL(X, Y_1), \|T\|=1$ and $x_1\in S_X$ be such that $\|Tx_1\|>1-\eta(\eps).$ Define $\tilde{T}\in \ELL(X, Y)$ by $\tilde{T}x=(Tx, 0).$ Then, $\|\tilde{T}\|=1, \|\tilde{T}x_1\|>1-\eta(\eps)$ and there exist $\tilde{S}\in \RSE(X, Y), x_0\in S_X$ such that
    \begin{equation*}
         \|\tilde{S}\|=1=\|\tilde{S}(x_0)\|, \quad \|\tilde{T}-\tilde{S}\|<\eps \text{ and } \|x_0-x_1\|<\eps.
    \end{equation*}
    Write $\tilde{S}x=(S_1x, S_2x)$ and choose $y^*=(y_1^*,y_2^*)\in S_{Y^*}$ such that $\re y^*(\tilde{S}x_0)=1.$ Define $S\in \ELL(X, Y_1)$ by
    \begin{equation*}
        Sx=S_1x+y_2^*(S_2x)\frac{S_1 x_0}{\|S_1x_0\|}.
    \end{equation*}
    It is shown in \cite[Proposition~2.7]{ACKLM2015} that 
    \begin{equation*}
        \|S\|=1=\|Sx_0\|, \quad \|T-S\|<\eps \text{ and } \|x_1-x_0\|<\eps.
    \end{equation*}
    Let $(x_n)\subseteq B_X$ be such that $\|Sx_n\|\longrightarrow 1=\|Sx_0\|.$ Then, 
    \begin{equation*}
    \begin{split}
        \left\| S_1x_n+y_2^*(S_2x_n)\frac{S_1x_0}{\|S_1x_0\|} \right\|&\leq \|S_1x_n\|+|y_2^*(S_2x_n)|\\
        &\leq \|S_1x_n\|+\|S_2x_n\|\\
        &=\|\tilde{S}(x_n)\|\leq 1,
    \end{split}
    \end{equation*}
    so $\|\tilde{S}x_n\|\longrightarrow 1=\|\tilde{S}\|$. As $\tilde{S}\in \RSE(X, Y)$, there exists $(\lambda_n)\subseteq \T$ such that $\lambda_n \tilde{S}x_n=\lambda_n(S_1x_n, S_2x_n)\longrightarrow \tilde{S} x_0=(S_1x_0, S_2x_0).$ Therefore, 
    \begin{equation*}
        \lambda_nSx_n\longrightarrow Sx_0.\qedhere
    \end{equation*} 
\end{proof}

The following can be proved as in \cite[Lemma~3.2]{DGMM2018} using Lemma \ref{lem:proyections-K} and Lemma \ref{prop:stability-K}.

\begin{lem}\label{lem:c0 -K}
    Let $X$ be a real or complex Banach space and let $Y$ be a real uniformly convex or a complex $\C$-uniformly convex Banach space, respectively. The following are equivalent.
    \begin{itemize}
        \item[i)] The pair $(c_0(X), Y)$ has the BPBp-RSE for compact operators.
        \item[ii)] There is a function $\eta:\R^+\longrightarrow \R^+$ such that the pairs $(\ell_\infty^m(X), Y)$ with $m\in\N$ have the BPBp-RSE for compact operators with the function $\eta.$
    \end{itemize}
    Moreover, when $\mathcal{K}(X, Y)=\ELL(X, Y)$ (in particular, if one of the spaces $X$ or $Y$ is finite dimensional), this happens when $(c_0(X), Y)$ or $(\ell_\infty(X), Y)$ has the BPBp-RSE.
\end{lem}

\begin{rem}
    If $Y$ is a uniformly convex or a $\C$-uniformly convex Banach space, then $Y$ does not contain an isomorphic copy of $c_0$, so in this cases we have $\KP(c_0, Y)=\ELL(c_0,Y)$ (see the proof of \cite[Theorem~2.4.11]{AK}, for example). Indeed, when $Y$ is uniformly convex in particular it is reflexive and, as reflexivity is inherited by closed subspaces, it cannot contain any copy of $c_0$. When $Y$ is $\C$-uniformly convex we can derive this fact from the proof of \cite[Proposition 2.2]{A2016}, for example. Therefore, in this case the BPBp-RSE for compact operators is just the BPBp-RSE. 
\end{rem}

\begin{theorem}
    Let $L$ be a locally compact Hausdorff topological space and let $Y$ be a complex Banach space. If $Y$ is $\C$-uniformly convex, then the subspace $\KP(C_0(L), Y)$ has the BPBp-RSE.
\end{theorem}
\begin{proof}
    If $Y$ is $\C$-uniformly convex, then from the previous remark $\KP(c_0,Y)=\ELL(c_0, Y)$ and it satisfies the BPBp-RSE by Theorems \ref{teo:L_infty char} or \ref{teo:c_0 sum} applied to the particular case of $c_0$. By Lemma \ref{lem:c0 -K}, there is a function $\eta:\R^+\longrightarrow \R^+$ such that the pairs $(\ell_\infty^m(X), Y)$ with $m\in\N$ have the BPBp-RSE for compact operators with the function $\eta.$ Using \cite[Lemma~3.4]{DGMM2018}, the hypotheses of Lemma \ref{lem:proyections-K} are satisfied for $X=C_0(L)$, so the subspace $\KP(C_0(L), Y)$ has the BPBp-RSE.
\end{proof}

This result together with Theorem \ref{prop:M summand} yield another characterisation.
\begin{cor}
    Let $Y$ be a complex Banach space and $(\Omega, \Sigma, \mu)$ a measure space. Then, the subspace $\KP(L_\infty(\mu),Y)$ has the BPBp if and only if $Y$ is $\C$-uniformly convex.
\end{cor}

We finish this section with another application of these techniques when the domain is an $\ell_1$-predual, that is, when $X^*$ is isometrically isomorphic to $\ell_1$. In this case we write $X^*\equiv \ell_1$.

\begin{theorem}
    Let $X$ be a Banach space such that $X^*\equiv \ell_1$ and let $Y$ be a Banach space. If $Y$ is uniformly convex or, in the complex case, if $Y$ is $\C$-uniformly convex, then the subspace $\KP(X, Y)$ has the BPBp-RSE. 
\end{theorem}
\begin{proof}
    By Theorem \ref{teo:c_0 real} and Corollary \ref{teo:L_infty char} the subspace $\ELL(c_0, Y)=\KP(c_0, Y)$ has the BPBp-RSE, so by Lemma \ref{lem:c0 -K} and the proof of \cite[Theorem~3.6]{DGMM2018} $X$ and $Y$ satisfy the conditions in Lemma \ref{lem:proyections-K}.
\end{proof}

We do not know if the condition of (complex) uniform convexity of $Y$ in the previous theorem is also necessary.

\subsection{Range spaces with ACK$_\rho$ structure}\label{sec:ACK}
The notions of ACK$_\rho$ structure and $\Gamma$-flat operators were introduced in \cite{CGKS2018} in order to abstract all the technicalities that make possible to prove a Bollobás-type theorem for Asplund operators acting to a space with the same structural properties of $C(K)$ and its uniform algebras. Indeed, the name ``ACK structure" comes from the words ``Asplund" and ``$C(K)$".

It is said that a Banach space $Y$ has \textit{ACK$_\rho$ structure} for some $\rho\in [0, 1)$ \cite[Definition 3.3]{CGKS2018} if there exists a 1-norming set $\Gamma\subseteq B_{Y^*}$ such that for every $\eps>0$ and every nonempty relatively $w^*$-open subset $U\subseteq \Gamma$ there exist a nonempty subset $V\subseteq U, y_1^*\in V, e\in S_Y$ and an operator $F\in \ELL(Y, Y)$  with the following properties:
\begin{enumerate}
    \item $\|Fe\|=\|F\|=1$,
    \item $y_1^*(Fe)=1$,
    \item $F^*y_1^*=y_1^*$,
    \item if we denote $V_1=\{y^*\in \Gamma :\|F^*y^*\|+(1-\eps)\|(Id_{Y^*}-F^*)(y^*)\|\leq 1\}$, then $|y^*(Fe)|\leq \rho$ for every $y^*\in \Gamma\backslash V_1$,
    \item $\dist(F^*y^*, \aconv\{0, V\})<\eps$ for every $y^*\in \Gamma$ and
    \item $|v^*(e)-1|\leq \eps$ for every $v^*\in V$.
\end{enumerate}

Banach spaces with ACK$_\rho$ structure include $C(K)$ for any compact Hausdorff topological space $K$ and its uniform algebras \cite[Corollary 4.6]{CGKS2018}, and spaces with property $\beta.$ Recall that a Banach space $Y$ satisfies property $\beta$ if there exist two sets $\{y_\alpha : \alpha\in \Lambda\}\subseteq S_Y , \{y_\alpha^* : \alpha\in \Lambda\} \subseteq S_{Y^*}$ and $0\leq \rho <1$ such that
 \begin{enumerate}
     \item $y_\alpha^*(y_\alpha)=1$,
     \item $|y_\alpha^*(y_\beta)|\leq \rho <1$ if $\alpha\neq \beta$ and
     \item $\|y\|=\sup_\alpha \{|y_\alpha^*(y)|\}$ for every $y\in Y.$
\end{enumerate}
If $Y$ satisfies this property then it has ACK$_\rho$ structure with the same value of $\rho$ and with the 1-norming set $\Gamma=\{y_\alpha^* : \alpha\in \Lambda\}$ \cite[Theorem 4.9]{CGKS2018}.

On the other hand, given two Banach spaces $X$ and $Y$ and a subset $\Gamma \subseteq Y^*$, an operator $T\in \ELL(X, Y)$ is said to be \textit{$\Gamma$-flat} \cite[Definition 2.8]{CGKS2018} if the restriction $T^*|_\Gamma : (\Gamma, w^*)\longrightarrow (X^*, \|\cdot\|_{X^*})$ is openly fragmented (see \cite[Definition 2.6]{CGKS2018}). The notation $\Fl_\Gamma (X, Y)$ will stand for the set of all $\Gamma$-flat operators between $X$ and $Y$. 

A linear subspace $\mathcal I\subseteq \Fl_\Gamma(X, Y)$ is said to be a \textit{$\Gamma$-flat ideal} \cite[Definition 3.3]{CGKS2018} if all elements of $\mathcal I$ are $\Gamma$-flat operators, $\mathcal I$ contains all the operators of finite rank and for every $T\in \mathcal I$ and every $F\in \ELL(Y, Y)$ the composition $F\circ T\in \mathcal I.$ It is known that the set $\mathcal A(X, Y)$ of Asplund operators is a $\Gamma$-flat ideal for every $\Gamma\subseteq Y^*$ \cite[Example A]{CGKS2018}. In particular, finite rank, compact and weakly compact operators are $\Gamma$-flat ideals for every $\Gamma\subseteq Y^*.$ It is also known that $\ELL(X, Y)=\Fl_\Gamma(X, Y)$ when $Y$ has property $\beta$ for the set $\Gamma=\{y_\alpha^* : \alpha\in \Lambda\}$ of the definition of this property (apply that $(\Gamma, w^*)$ is discrete and \cite[Example C]{CGKS2018}). 

In \cite[Theorem 3.4]{CGKS2018} it was proved, in particular, that if $Y$ has ACK$_\rho$ structure with corresponding 1-norming set $\Gamma\subseteq Y^*$ then any $\Gamma$-flat ideal $\mathcal I\subseteq \Fl_\Gamma(X, Y)$ satisfies the BPBp for any Banach space $X$. This generalizes some previously known results in the literature: if $Y$ has property $\beta$, then $(X, Y)$ has the BPBp for every $X$ \cite[Theorem 2.2]{AAGM2008}, any subideal $\mathcal I(X, Y)\subseteq \mathcal A(X, Y)$ has the BPBp whenever $Y=C(K)$ \cite[Corollary 2.5]{ACK2011} and whenever $Y$ is a uniform algebra $\mathfrak{U}\subseteq C(K)$ \cite[Theorem 3.6]{CGK2013}.

However, the RSE version of all these results fails to be true, as shown in Example \ref{ex: ell_1^2}: in the real case the space $\ell_1^2\equiv \ell_\infty^2$ satisfies property $\beta$, it is a $C(K)$-space and every operator in $\ELL(\ell_1^2, \ell_1^2)$ is Asplund, but the pair $(\ell_1^2, \ell_1^2)$ does not satisfy the BPBp-RSE by Proposition \ref{prop:l1sum}. Thus, in order to obtain some analogous results to the previously mentioned ones we need some conditions on the domain. For the case of the BPBp-ASE, in \cite[Theorem 3.13]{JMR2023} the authors showed that property [P] of the domain is enough. By Proposition \ref{prop:property P is necessary} it is also necessary, so in this case we have a complete characterisation.

\begin{theorem}\label{teo:ACK ASE}
    Let $X, Y$ be Banach spaces such that $Y$ has ACK$_\rho$ structure with the corresponding 1-norming set $\Gamma\subseteq B_{Y^*}$, and let $\mathcal I(X, Y)\subseteq \Fl_\Gamma(X, Y)$ be a $\Gamma$-flat ideal. Then $\mathcal I(X, Y)$ satisfies the BPBp-ASE if and only if $X$ has property [P]. 
\end{theorem}

The following definition is motivated by property [P], and the name stands for ``P-weak". Recall that if $x^*\in \smo(X^*)$ attains its norm at $x_0$ and $$\re x^*(x_n) \longrightarrow \|x^*\|=x^*(x_0)$$ then $x_n$ converges to $x_0$ in the weak topology of $X.$

\begin{defi}
     A Banach space $X$ is said to satisfy property [P$_w$] if for every $\eps >0$ there exists $\eta (\eps)>0$ such that for every $x^*\in S_{X^*}$ and every $x\in S_X$ satisfying 
    \begin{equation*}
        \re x^*(x)>1-\eta(\eps)
    \end{equation*}
    there exist $y^*\in\smo(X^*)$ and $y\in S_X$ such that
    \begin{equation*}
        y^*(y)=1, \quad \|x^*-y^*\|<\eps \quad\text{and}\quad \|x-y\|<\eps.
    \end{equation*}
\end{defi} 

\begin{example}
Every strictly convex Banach space satisfies property [P$_w$]; just observe that such spaces satisfy $\NA(X)\subseteq \smo(X^*)$ and apply the classical Bishop-Phelps-Bollobás theorem. Moreover, any Banach space $X$ such that the norm of $X^*$ is G\^{a}teaux differentiable at every point but nowhere Fr\'echet differentiable satisfies property [P$_w$] but not property [P] (indeed, it does not have any strongly exposed points). This shows that property [P$_w$] is strictly weaker than property [P].
\end{example}  

\begin{rem}\label{rem: uniform modulus P-Pw}
If a Banach space $X$ satisfies property [P] or [P$_w$], it can be proved applying the classical Bishop-Phelps-Bollobás theorem that the function $\eta$ in the definition can be taken to be $\eta(\eps)=\frac{\eps^2}{2}$. In view of the proof of \cite[Theorem 3.13]{JMR2023}, this shows that the function $\eta$ of the Bollobás approximation by ASE operators in this result only depends on the constant $\rho$ of the ACK$_\rho$ structure of the range.
\end{rem}

The key to use property [P$_w$] to approximate bounded linear operators by operators in $\RSE(X, Y)$ is contained in the following lemma. The statement is slightly different from \cite{CDFJM2025}, but it is exactly what is proved there.
\begin{lem}\cite[Lemma~4.2]{CDFJM2025}\label{lem:smo}
    Let $X, Y$ be Banach spaces and $\mathcal{I}\subseteq \ELL(X, Y)$ any linear subspace containing all rank-one operators such that $\mathcal I\subseteq \mathcal{DP}(X,Y)$. Suppose that $T\in \mathcal{I}(X, Y)$ satisfies $\|T\| = \|Tx_0 \|=1$ and $y_0^*\in S_{Y^*}$ is such that $\re y_0^*(Tx_0)=\|T\|$ and $T^* y_0^* \in \smo({X^*})$. Then, given $\varepsilon>0$, there is $S \in \RSE (X,Y)\cap \mathcal I(X, Y)$ such that $\|Sx_0\| = \|S\|=1$ and $\|S-T\|<\varepsilon$.
\end{lem}

Now we are able to extend \cite[Theorem 3.4]{CGKS2018} to the BPBp-RSE for a wide class of operators including the ideals of compact, weakly compact and finite rank, assuming that the domain satisfies property [P$_w$]. 
\begin{theorem}\label{teo:ACK RSE}
    Let $X, Y$ be Banach spaces such that $Y$ has ACK$_\rho$ structure with the corresponding 1-norming set $\Gamma\subseteq B_{Y^*}$ and let $\mathcal I(X, Y)\subseteq \Fl_\Gamma(X, Y)\cap \DP(X, Y)$ be a $\Gamma$-flat ideal. If $X$ has property [P$_w$], then $\mathcal I(X, Y)$ satisfies the BPBp-RSE. Moreover, the funcion $\eta(\eps)$ only depends on the constant $\rho$ of the ACK$_\rho$ structure of $Y$.
\end{theorem}
The proof is an adaptation of the argument in the cited result in a similar way as it was done in \cite[Theorem 3.13]{JMR2023}. The following can be obtained arguing like in \cite[Lemma~2.9]{CGKS2018} using property [P$_w$] instead of the Bishop-Phelps-Bollobás theorem. Notice that the constants are independent of the spaces $X, Y$ by Remark \ref{rem: uniform modulus P-Pw}.
\begin{lem}\label{lem: basic lemma Gamma flat}
    Let $X$ be a Banach space satisfying property [P$_w$] and let $Y$ be a Banach space. Let $\Gamma\subseteq B_{Y^*}$ be a 1-norming set, $T\in \ELL(X, Y)$ be a $\Gamma$-flat operator with $\|T\|=1, \eps>0$ and $x_0\in S_X$ such that $\|Tx_0\|>1-\frac{\eps^2}{2}$. Then, for every $r>0$ there exist
    \begin{itemize}
        \item[i)] A w*-open set $U_r\subseteq Y^*$ with $U_r\cap \Gamma\neq \emptyset$ and
        \item[ii)] $x_r^*\in \smo(X^*)\cap S_{X^*}$ and $u_r\in B_X$ such that $|x_r^*(u_r)|=1$, $\|T^*z^*-x_r^*\|<r+\eps+\frac{\eps^2}{2}$ and $\|u_r-x_0\|<\eps$ for every $z^*\in U_r\cap \Gamma.$
    \end{itemize}
\end{lem}

\begin{proof}[Proof of Theorem \ref{teo:ACK RSE}]
    Given $\eps>0$, fix $0<\eps_0<\eps$ and take $\eps_1>0$ such that
    \begin{equation*}
        \max \left\{\eps_1, 4\left(\eps_1+\frac{\eps_1^2}{2}+\frac{2(\eps_1+\frac{\eps_1^2}{2})}{1-\rho+\eps_1+\frac{\eps_1^2}{2}}\right)\right\}<\eps_0.
    \end{equation*}
Take $r>0$ and $0<\eps_2<\frac{\eps}{3}$ so that $3\eps_2+r<\eps_1+\frac{\eps_1^2}{2}$. Let $\eta(\eps)=\frac{\eps_1^2}{2}.$

Now, let $T\in \mathcal{I}(X, Y)$ be a $\Gamma$-flat operator and $x_0\in S_X$ such that $\|T\|=1$ and $\|Tx_0\|>1-\eta(\eps)$. By Lemma \ref{lem: basic lemma Gamma flat}, there exist
\begin{itemize}
    \item[i)] A $w^*$-open set $U_r\subseteq Y^*$ such that $U_r\cap \Gamma \neq \emptyset$
    \item[ii)] $x_r^*\in \smo(X^*)\cap S_{X^*}$ and $u_r\in B_X$ such that $|x_r^*(u_r)|=1$, $\|T^*z^*-x_r^*\|<r+\eps_1+\frac{\eps_1^2}{2}$ and $\|u_r-x_0\|<\eps_1$ for every $z^*\in U_r\cap \Gamma.$
\end{itemize}

On the other hand, as $Y$ has ACK$_\rho$ structure, we can obtain $V\subseteq U_r\cap \Gamma, y_1^*\in V, e\in S_Y, F\in \ELL(Y, Y)$ satisfying properties (a)-(f) in the definition of ACK$_\rho$ structure. Define $S\in \ELL(X, Y)$ by
\begin{equation*}
    Sx :=x_r^*(x)F(e)+(1-\delta)(Id_Y-F)T(x), x\in X,
\end{equation*}
     where $\delta\in (\eps_2, 1)$ is chosen so that $\|S\|\leq 1$ (see \cite[Lemma~3.5]{CGKS2018}). Clearly, $S\in \mathcal{I}(X, Y)$. Using that $y_1^*(Fe)=1$ and $F^*y_1^*=y_1^*$ we have, for every $x\in X,$
     \begin{equation*}
         S^*y_1^*(x)=x_r^*(x)+(1-\delta)y_1^*(Tx)-(1-\delta)y_1^*(Tx)=x_r^*(x),
     \end{equation*}
so $S^*y_1^*=x_r^*\in \smo(X^*)$ and
\begin{equation*}
    1=|x_r^*(u_r)|=|y_1^*(Su_r)|\leq \|Su_r\|\leq 1.
\end{equation*}
Arguing as in \cite[Lemma~3.5]{CGKS2018}, we have that $\|S-T\|<\frac{\eps}{2}.$ On the other hand, by Lemma \ref{lem:smo} we can find $R\in \RSE(X, Y)\cap \mathcal I(X, Y)$ such that $\|R\|=1=\|R(u_r)\|$ and $\|S-R\|<\frac{\eps}{2}$, so $\|T-R\|<\eps.$ As $\|u_r-x_0\|<\eps_1<\eps$ we have finished. For the last statement, notice that the function $\eta(\eps)$ defined above only depends on the constant $\rho$ of the ACK$_\rho$ structure of $Y$ in virtue of Remark \ref{rem: uniform modulus P-Pw}.
\end{proof}

We shall collect some of the consequences of Theorems \ref{teo:ACK ASE} and \ref{teo:ACK RSE}.

\begin{cor}
    Let $X, Y$ be Banach spaces and suppose that $Y$ has ACK$_\rho$ structure. Then,
    \begin{enumerate}
        \item If $X$ has property [P$_w$], then the subspaces $\FR(X, Y)$ and $\KP(X, Y)$ have the BPBp-RSE.
        \item If $X$ has property [P], then the subspaces $\FR(X, Y), \KP(X, Y)$ and $\mathcal{W}(X, Y)$ have the BPBp-ASE.
    \end{enumerate}
\end{cor}

These results yield another interesting consequence for compact and finite rank operators when the range is isometric to $L_1$.

\begin{cor}
    Let $X$ be a Banach space satisfying property [P$_w$] and $Y$ an $L_1$-predual. Then, for every $\eps>0$ there is $\eta(\eps)>0$ such that if $T\in S_{\KP(X, Y)}$ and $x_0\in S_X$ are such that $\|Tx_0\|>1-\eta(\eps)$ then there exist $S\in\RSE(X, Y)\cap\FR(X, Y)$ and $x\in S_X$ such that 
    \begin{equation*}
        \|Sx\|=\|S\|=1, \quad \|x-x_0\|<\eps \text{ and } \|S-T\|<\eps.
    \end{equation*}
    Moreover, the same $\eta(\eps)$ works for every $X$ with property [P$_w$]. In particular, the subspaces$\KP(X, Y)$ and $\FR(X, Y)$ have the BPBp-RSE.
\end{cor}
The proof is based on the one of \cite[Lemma~3.4]{JW1979} and follows an analogous argument to \cite[Theorem~4.2]{ABCCKLLM014}.
\begin{proof}
    For every $n\in \N$, the space $\ell_\infty^n$ has ACK$_\rho$ structure with $\rho=0$ and the corresponding 1-norming set $\Gamma=\{e_i^*:i=1, \dots, n\}$ and every operator in $\ELL(X, Y)$  is $\Gamma$-flat, so by Theorem \ref{teo:ACK RSE} there is $\eta'$ such that $(X, \ell_\infty^n)$ satisfies the BPBp-RSE with the same function $\eta'$ for all $n\in \N.$
    
    Fix $\eps>0$ and let $T\in \KP(X, Y)$ with $\|T\|=1$ and $x_0\in S_X$ such that $\|Tx_0\|>1-\frac{\eta'(\eps/2)}{2}.$ Arguing like in \cite[Lemma~3.4]{JW1979}, there exists a subspace $E$ of $Y$ which is isometric to $\ell_\infty^m$ for some $m\in \N$ and a projection $P:Y\longrightarrow E\subseteq Y$ with $\|P\|=1$ such that $\|T-P\circ T\|<\min \{\frac{\eta'(\eps/2)}{2}, \frac{\eps}{4}\}.$ Then, 
    \begin{equation*}
    \begin{split}
        \|PT(x_0)\|&=\|Tx_0-Tx_0+PT(x_0)\|\\
        &\geq \|Tx_0\|-\|Tx_0-PT(x_0)\|\\
        &>1-\frac{\eta'(\eps/2)}{2}-\frac{\eta'(\eps/2)}{2}\\
        &=1-\eta'(\eps/2).
    \end{split} 
    \end{equation*}
    Therefore, the operator $R=\frac{PT}{\|PT\|}$ satisfies
    \begin{equation*}
        \|Rx_0\|\geq \|PTx_0\|>1-\eta'(\eps/2).
    \end{equation*}
    As the pair $(X, E)$ has the BPBp-RSE with the function $\eta'$, there exist $S\in\RSE(X, E)$ and $x\in S_X$ such that 
    \begin{equation*}
        \|Sx\|=\|S\|=1, \quad \|R-S\|<\frac{\eps}{2} \text{ and } \|x-x_0\|<\frac{\eps}{2}.
    \end{equation*}
    $S\in\RSE(X, E)$ can be viewed as a finite rank operator $S:X\longrightarrow Y$ and
    \begin{equation*}
    \begin{split}
        \|T-S\|&\leq \|T-PT\|+\|PT-R\|+\|R-S\|\\
        &<\frac{\eps}{4}+1-\|PT\|+\frac{\eps}{2}\\
        &<\frac{\eps}{4}+\frac{\eps}{4}+\frac{\eps}{2}=\eps. \qedhere
    \end{split}
    \end{equation*}
\end{proof}
\begin{rem}
    The previous result is also valid for the BPBp-ASE assuming that $X$ has property [P] and using Theorem \ref{teo:ACK ASE}.
\end{rem}

\section{Further discussion and open questions}\label{sec:discussion}

We can define the notions of universality of the domain and range for the BPBp-RSE and the BPBp-ASE in an analogous way as it was done in \cite{ACKLM2015} for the classical Bishop-Phelps-Bollobás property.
\begin{defi}
    Let $X$ and $Y$ be Banach spaces. 
    \begin{enumerate}
        \item We say that $X$ is a universal BPBp-RSE  domain space (resp., BPBp-ASE) if for every Banach space $Z$, the pair $(X, Z)$ has the BPBp-RSE (resp., BPBp-ASE).
       \item  We say that $Y$ is a universal BPBp-RSE  range space (resp., BPBp-ASE) if for every Banach space $Z$, the pair $(Z, Y)$ has the BPBp-RSE (resp., BPBp-ASE).
    \end{enumerate}
\end{defi}

We know from Proposition \ref{prop:property P is necessary} that there is no universal BPBp-ASE range space and that property [P] is a necessary condition for being a universal BPBp-ASE domain. However, this is the only necessary condition we are able to get.
\begin{question}
    Are there more necessary conditions for a Banach space $X$ to be a universal BPBp-ASE domain? In particular, does this imply that $X$ is uniformly convex?
\end{question}

Regarding to the BPBp-RSE, we can derive that the denseness of $\str(B_X)$ is a necessary condition for being a universal domain also in this case. For this, we take advantage of the following fact.
\begin{prop}[\mbox{\cite[Remark~2.4]{CDFJM2025}}]
    Let $X, Y$ be Banach spaces and $T\in \ELL(X, Y)$ a monomorphism. Then, $T\in \RSE(X, Y)$ if and only if $T\in \ASE(X, Y).$
\end{prop}

\begin{cor}
    If $X$ is a Banach space such that $(X, X)$ has the BPBp-RSE, then $\str(B_X)$ is dense in $S_X.$ In particular, this is a necessary condition for being a universal BPBp-RSE domain (or range space).
\end{cor}
\begin{proof}
    Let $x_0\in S_X$ and $\eps>0$ and take $T=Id: X\longrightarrow X.$ Then, $1=\|T\|=\|x_0\|=\|Tx_0\|>1-\eta(\eps)$. Therefore, there exist $S\in \RSE(X, Y), x_1\in S_X$ such that $$\|Sx_1\|=1=\|S\|, \quad \|S-T\|<\eps \text{ and } \|x_1-x_0\|<\eps.$$
    As the set of isomorphisms from $X$ to $X$ is open, taking $\eps$ sufficiently small we can assure that $S$ is an isomorphism and therefore $S\in\ASE(X, X),$ so $x_1\in \str(B_X)$ as desipurple.
\end{proof}

\begin{question}
    If $X$ is a universal domain for BPBp-RSE, does this imply that $X$ is a universal domain for BPBp-ASE?
\end{question}

In what comes to the range spaces, the only universal BPBp-RSE range space we know is the field of scalars, $\K$, and no infinite dimensional Banach space $Y$ can be a universal BPBp-RSE range as, in such case, there exists a Banach space $X$ for which $\RSE(X, Y)$ is not dense in $\ELL(X, Y)$ (see \cite[Theorem 3.1]{CDFJM2025}). Moreover, Theorem \ref{teo:L_1 char} yields the following.

\begin{cor}
    Let $Y$ be a Banach space. If $Y$ is a universal BPBp-RSE range space, then $Y$ is uniformly convex and finite dimensional.
\end{cor}

Observe that it is an old open problem wether every finite dimensional Banach space satisfies property B of Lindenstrauss (i.e., $\NA(X, Y)$ is dense in $\ELL(X, Y)$ for every $X$), and that this is not known even for the case $Y=\ell_2^2.$

It is worth pointing out a consequence of Lemma \ref{prop:stability-K} which can be proved using the same argument as in \cite[Corollary~2.2]{ACKLM2015}.

\begin{cor}
    Let $X$ and $Y$ be Banach spaces.
    \begin{itemize}
        \item If $X$ is a universal BPBp-RSE domain, there is a function $\eta_{X}\colon (0,1)\longrightarrow \R^+$ such that for every Banach space $Y$, the pair $(X, Y)$ has the BPBp-RSE with function $\eta_X$.
        \item If $X$ is a universal BPBp-RSE range, there is a function $\eta_{Y}\colon (0,1)\longrightarrow \R^+$ such that for every Banach space $X$, the pair $(X, Y)$ has the BPBp-RSE with function $\eta_Y$.
    \end{itemize}  
\end{cor}

We conclude the paper with some more open questions about the topic of the paper. The first one is naturally motivated by the results in Sections \ref{sec: M summand complex} and \ref{sec: M summand real}. Indeed, when $K$ is connected, $C(K)$ does not have any $M$-summands and therefore Theorems \ref{prop:M summand} and \ref{prop:l infty sum real} cannot be applied to the pair $(C(K), Y).$
\begin{question}
    If the pair $(C[0,1], Y)$ has the BPBp-RSE in the complex case, does this imply that $Y$ is $\C$-uniformly convex? In the real case, does this imply that $Y$ is uniformly convex?
\end{question}

All along the paper we have focused mainly on the BPBp-RSE. However, we think that it may be of interest to study the BPBp-ASE itself. The first natural question is whether property [P] of the domain can be enough to get the BPBp-ASE under the assumption that the pair $(X, Y)$ has the BPBp. In other words,

\begin{question}
Is the BPBp-ASE of a pair $(X, Y)$ equivalent to property [P] of the domain together with the classical BPBp of $(X, Y)$?
\end{question}

Another natural question is motivated by the fact that in all the cases in which we know that a pair $(X, Y)$ satisfies the BPBp-RSE, this property is also satisfied by the subspace $\KP(X, Y)$, and we are not aware of any example in which the compact operators between $X$ and $Y$ have the BPBp-RSE but the pair $(X, Y)$ does not have it. Let us point out that it is known that the classical BPBp for compact operators does not imply the BPBp in general (for instance, the pair $(L_1[0,1], C[0,1])$ satisfies the BPBp for compact operators (see \cite[R2, p. 380]{CGK2013}) but $\NA(L_1[0,1], C[0,1])$ is not dense in $\mathcal{L}(L_1[0,1], C[0,1])$ as proved in \cite{S1983}) and, up to our knowledge, it is still an open question whether the BPBp implies the BPBp for compact operators.

\begin{question}
    Is the BPBp-RSE equivalent to the BPBp-RSE for compact operators?
\end{question}

We finish with an open problem about the classical BPBp. A positive answer to this question would yield automatically the BPBp-RSE in virtue of Corollary \ref{cor:strBPBp}.
\begin{question}
    Suppose that $X$ has finite dimension and $Y$ is uniformly convex. Does the pair $(X, Y)$ have the BPBp(-RSE)?
\end{question}

\section*{Acknowledgments}
The author wishes to express her gratitude to Sun Kwang Kim and Rubén Medina for suggesting the topic of this paper. She is also grateful to Mingu Jung and Óscar Roldán for providing valuable information and relevant references concerning specific parts of the work. This paper forms part of the author’s PhD thesis, supervised by Sheldon Dantas and Miguel Martín, to whom she extends her thanks for their continuous guidance and support throughout the writing process.

The author was partially supported by the grant PRE2022-103590 funded by MICIU/AEI/10.13039/501100011033 and by “ESF+”, by the grant PID2021-122126NB-C31
funded by MICIU/AEI/10.13039/501100011033 and ERDF/EU, and by ``Maria de Maeztu'' Excellence Unit IMAG, funded by MICIU/AEI/10.13039/501100011033 with reference CEX2020-001105-M.

\bibliographystyle{alpha}

\end{document}